\documentclass[11pt]{article}

\usepackage{amsmath, amssymb, amsthm, color}
\usepackage{hyperref, url}
\usepackage{graphicx}
\usepackage{ulem}
\usepackage{tikz}
\usepackage{epstopdf}
\graphicspath{{./figs/}}
\theoremstyle{thmstyletwo}%
\newtheorem{example}{Example}[section]
\newtheorem{remark}{Remark}%
\newtheorem{corollary}{Corollary}[section]
\newtheorem{proposition}{Proposition}
\newtheorem{theorem}{Theorem}
\newtheorem{assumption}{Assumption}

\theoremstyle{thmstylethree}%
\newtheorem{definition}{Definition}%
\DeclareRobustCommand{\RR}{\mathbb R}
\DeclareRobustCommand{\setS}{\mathbb S}

\DeclareRobustCommand{\UU}{\mathbf U}
\DeclareRobustCommand{\CC}{\mathbf C}

\newtheorem{lemma}{Lemma}[section]

\newcommand{\interior}{{\mbox{\rm int}}}

\newcommand{\dist}{{\mbox{\rm dist}}}

\begin{document}

\title{ An Augmented Lagrangian Approach to Conically Constrained Non-monotone Variational Inequality Problems }

%\title{An Augmented Lagrangian Decomposition Method for Non-monotone Constrained Variational Inequality Problems}
%[Augmented Lagrangian Decomposition Method for Non-monotone Constrained VIP]

\author{
Lei Zhao\thanks {Institute of Translational Medicine and National Center for Translational Medicine, Shanghai
Jiao Tong University, 200030 Shanghai, China ({\tt l.zhao@sjtu.edu.cn}).}
\and
Daoli Zhu\thanks {Antai College of Economics and Management and Sino-US Global Logistics
Institute, Shanghai Jiao Tong University, 200030 Shanghai, China
({\tt dlzhu@sjtu.edu.cn}); and School of Data Science, The Chinese University of Hong Kong, Shenzhen, Shenzhen 518172, China}
\and Shuzhong Zhang\thanks {Department of Industrial and Systems Engineering, University of Minnesota, Minneapolis, MN 55455, USA
({\tt zhangs@umn.edu})}}
%\footnotetext[1]{Acknowledgment: this research was supported by NSFC grants 71471112 and 71871140.}

\maketitle

\begin{abstract}
In this paper we consider a non-monotone (mixed) variational inequality model with (nonlinear) convex conic constraints.
Through developing an equivalent Lagrangian function-like primal-dual saddle-point system for the VI model in question, we introduce an
augmented Lagrangian primal-dual method, to be called ALAVI in the current paper, for solving a general constrained VI model. Under an assumption, to be called the primal-dual variational coherence condition in the paper, we prove the convergence of ALAVI. Next, we show that many existing generalized monotonicity properties are sufficient---though by no means necessary---to imply the above mentioned coherence condition, thus are sufficient to ensure convergence of ALAVI. Under that assumption, we further show that ALAVI has in fact an $o(1/\sqrt{k})$ global rate of convergence where $k$ is the iteration count. By introducing a new gap function, this rate further improves to be $O(1/k)$ if the mapping is monotone. %ergodic rate of convergence respectively.
Finally, we show that under a metric subregularity condition, even if the VI model may be non-monotone the local convergence rate of ALAVI improves to be linear. Numerical experiments on some randomly generated highly nonlinear and non-monotone VI problems show practical efficacy of the newly proposed method. 

%even though the mapping is non-monotone.

%primal-dual variational coherent-condition (weaker than monotonicity) for (VIP), we prove that the entire sequence generated by ALAVI converges to one solution of (VIP). Furthermore, we introduce a KKT stationary measure and a new gap function for (VIP), and provide the iteration complexity ($O(1/\sqrt{k})$ nonergodic and $O(1/k)$ ergodic rate). Finally, we show that ALAVI exhibits R-linear convergence rate under metric subregularity of KKT mapping of (VIP).
\end{abstract}

\vspace{1cm}



\section{Introduction}\label{intro}

The ability to solve variational inequality (VI) problems resides at the core of mathematical programming, minimax saddle point models, equilibrium and complementarity problems, computational games, and economics theory. In terms of algorithm design for (VI), much of the attention has been focused on the case where the underlying mapping is {\it monotone}. This paper aims to present a new algorithm for {\it non-monotone}\/ and {\it constrained}\/ variational inequality problems in a general form. Specifically, let us consider a continuous mapping $G$ from $\RR^n$ to $\RR^n$; $J$ is a continuous function on $\mathbf{U}\subseteq \RR^{n}$ possibly nonsmooth; $\UU$ is a closed convex subset of $\RR^n$; $\Theta(x)$ is a continuous mapping from $\RR^n$ to $\RR^m$ and $\mathbf{C}$ is a nonempty closed convex cone in $\RR^{m}$.
% with vertex at the origin, that is, $\alpha\mathbf{C}+\beta\mathbf{C}\subseteq \mathbf{C}$, for $\alpha,\beta\geq 0$.
With the above setting in mind, in this paper we consider the following generic
mixed variational inequality problem with nonlinear conic constraint:
\begin{equation}\label{Prob:VIP}
\begin{aligned}
\mbox{{\rm(VIP)}}\qquad&\mbox{Find $u^*\in\Gamma$ such that}\\
&\langle G(u^*),u-u^*\rangle+J(u)-J(u^*)\geq0,\quad\forall u\in\Gamma,\\
&\mbox{ where } \Gamma=\left\{u\in\UU\mid\Theta(u)\in-\mathbf{C}\right\}.
\end{aligned}
\end{equation}
%{\color{brown} (We think hemi-VI is more general than the mixed VI problem we considered in this paper. Such modification may cause unnecessary trouble.)}
Note that the above model also relates to the so-called hemi-variational inequality problem (cf.~\cite{FPSL14,P93}).
%In fact, hemi-VI is a useful unification, allowing further model extensions;
%The model admits further generalization; e.g.\ cf.~\cite{FPSL14} for a host of monotone VI-constrained bilevel problems.}
%In this paper, we focus on the hemi-VI model (VIP) as our baseline variational inequality problem to be solved.}
Remark that any convex constraints can be reformulated via homogenization into the above form; see e.g.~\cite{Z04}. Also, $\Theta$ is not necessarily affine either; for instance, negative matrix logarithmic function $\Theta$ and a PSD matrix cone $\mathbf{C}$ would also be a valid example of (VIP) defined by~\eqref{Prob:VIP}. As we will discuss shortly, the regularization function $J$ is assumed to be convex though not necessarily smooth. In this paper we focus on the case where $G$ is non-monotone. Therefore, the above model (VIP) is essentially a non-monotone VI problem. Let us denote $\setS_{\mbox{\tiny VIP}}$ to be the set of all solutions to (VIP). We assume $\setS_{\mbox{\tiny VIP}}\not=\emptyset$ throughout this paper.

It is well known that constrained variational inequality formulation unifies many problem types in mathematical programming such as constrained optimization, constrained min-max problem (cf.~\cite{ZWHZ21}), and many more.
%{\color{blue}, \sout{and generalized Nash equilibrium with shared constraints (cf.~\cite{GenN}), peer-to-peer electricity market (cf.~\cite{P2P}), rider sharing equilibrium (cf.~\cite{RideShare}\cite{RideShare2})}}, among many others.
An excellent reference on the theory and solutions methods for finite-dimensional variational inequality problems can be found in Facchinei and Pang~\cite{FP07}.
%Thus, (VIP) is gaining much attention in the literature and variety of algorithms is developed for their solution.
%In order to efficiently solve large scale (VIP), we need to develop the decomposition methods for their solution.
Though much of the theory and solution methods for VI have been mainly focussing on monotone VI problems, there are studies on non-monotone VI's as well. Among those,
%Along the line of decomposition-based approach,
Zhu and Marcotte (section 6 of~\cite{Zhu96}) proposed a Uzawa-type decomposition method for (VIP) with explicit constraints. However, the strong monotonicity of $G$ is required to guarantee convergence in their approach. In order to relax the strong monotonicity of $G$ to co-coercivity, Zhu~\cite{Zhu03} proposed an augmented Lagrangian decomposition
%\sout{augmented Lagrangian decomposition}}
method for (VIP).  %respectively.
The current paper continues that line of investigations.

We shall also remark that the resurgent interest on non-monotone VI's has been partly motivated by nonconvex-nonconcave min-max saddle-point models, which is relatable to, e.g., deep learning neural networks; see a recent survey by Razaviyayn {\em et al.}~\cite{RLNHSH20}.
%Along totally different lines of investigations,
%In recent machine learning research, various nonconvex-nonconcave min-max models are proposed and studied.
%The underlying engines %main work horses
%are often the methodologies designed for solving monotone (or nonmonotone) variational inequalities (without or with explicit constraints).
Regarding VI models without co-coercivity,
recently Malitsky and Tam~\cite{FBS} and Malitsky~\cite{Golden} developed an approach for %the method for solving
monotone inclusion problem and VI without explicit constraints.
Kanzow and Steck~\cite{ALM1} considered (VIP) (with $J=0$) in Banach spaces. To solve (VIP), they proposed a two-loop type algorithm, where the inner loop consists of an inexact zero-finding process for a nonlinear system generated from the augmented Lagrangian. We shall remark here that their inner loop process does not preserve separability for large scale problems, even if the initial problem is separable.
For the interior point approach to VI, Goffin {\it et al.}~\cite{GMZ97} and Marcotte and Zhu~\cite{MZ01} developed the analytic center cutting plane framework to solve (VIP) where $\Gamma$ is a polyhedron and $G$ is pseudo-monotone$^+$ or quasi-monotone. Recently, Yang {\em et al.}~\cite{YJ221} and Chavdarova {\it et al.}~\cite{YJ222} developed an interior point approach to solve monotone VI with nonlinear inequality and linear equality constraints, and they established an iteration complexity analysis. Very recently, Lin and Jordan~\cite{LJ22} developed an iteration complexity analysis for an algorithm (which they termed Perseus) using the high-order derivative information to solve non-monotone VI, assuming the so-called Minty solutions exist. Though our approaches are entirely different, in this paper we also base our method on the notion of primal-dual Minty solutions for a reformulated hemi-VI system. In particular, we reformulate (VIP) into a {\em mixed}\/ Minty-type system, and then design a primal-dual iterative algorithm (based on an augmented Lagrangian) to search for its solution.
%In this work, we aim to develop the decomposition methods to solve non-monotone (VIP)~\eqref{Prob:VIP}. Towards that end, we introduce an additional variable and develop one Modified Augmented Lagrangian Decomposition Method (ALAVI) for solving (VIP). In ALAVI, the primal update process is decomposable, where it incorporates the linearization of augmented Lagrangian and a proximal term.
Our approach is an extension of our recent work (\cite{ZZZ22}) for solving non-convex optimization and an early work on the co-coercive (VIP)~\cite{Zhu03}; the latter two methods are based on the augmented Lagrangian approach which %turn out 
ensures 
strong global convergence with simple primal-dual updates.
In summary, the main contributions of this paper are as follows:
\begin{description}
\item[(i)] We develop a Lagrangian function based primal-dual system for (VIP) and design an
%{\color{blue}\sout{modified}}
augmented Lagrangian
%{\color{blue}\sout{decomposition}}
method, which we shall call ALAVI, for solving (VIP).

\item[(ii)] We prove that the global convergence of ALAVI under what we call the {\it primal-dual variational coherence condition}\/ %which is implied by many generalized monotonicity conditions,
for (VIP).

\item[(iii)] We introduce a suite of generalized monotonicity conditions and show that they are all sufficient conditions for the above primal-dual variational coherence condition to hold.

\item[(iv)] Under the primal-dual variational coherence condition, we introduce a KKT stationary measure for (VIP), by which we show ALAVI to have an $o(1/\sqrt{k})$ iteration complexity (in the non-ergodic sense).

\item[(v)]  In the case $G$ is monotone, we introduce a new gap function for (VIP) and show that ALAVI has an $O(1/k)$ iteration complexity in the ergodic sense.
%rate of ALAVI under monotonicity of $G$.

\item[(vi)] Finally, if the KKT mapping of (VIP) satisfies a metric subregularity condition while $G$ is in general non-monotone, then we further establish an R-linear rate of ALAVI for (VIP).
\end{description}

The remainder of this paper is organized in the following way. In Section~\ref{sec:pre}, we provide some preliminaries that are necessary for our analysis. Specifically, Subsection~\ref{sec:Lagrangian} is devoted to a Lagrangian function theory for (VIP). In Subsection~\ref{subsec:minty}, we introduce the (mixed) Minty variational inequality system and the primal-dual variational coherence condition for (VIP). Next, in Section~\ref{sec:monotone}, we introduce a suite of generalized monotonicity conditions on the mapping (and the objective) and prove that they are sufficient for the primal-dual variational coherence condition to hold.  In Section~\ref{sec:ALAVI_convergence}, we present a new method called
Augmented Lagrangian Approach to Variational Inequalities %\sout{Modified Augmented Lagrangian Decomposition}}
(abbreviated as ALAVI henceforth) to solve (VIP), and provide an analysis for the global convergence of ALAVI for (VIP) under the primal-dual variational coherence condition. In Section~\ref{sec:rate}, we introduce a KKT stationary measure for (VIP) and provide an iteration complexity analysis for ALAVI based on that notion. Finally, in Section~\ref{sec:linear} an R-linear convergence rate under the metric subregularity is shown to hold for ALAVI.

\section{Preliminaries}\label{sec:pre}
To pave the ground for our discussion, let us start with the notations and assumptions.
%This section is dedicated to give some notations and assumptions as the setting we consider latter.
\begin{definition} [Conic convexity of mapping $\Theta(u)$] Let $\mathbf{C}$ be a closed convex cone. The mapping $\Theta(u)$ is called $\mathbf{C}$-convex if
\begin{equation*}
\forall u,v\in\mathbf{U},\quad\forall\alpha\in [0,1],\quad\Theta(\alpha u+(1-\alpha)v)-\alpha\Theta(u)-(1-\alpha)\Theta(v)\in - \mathbf{C}.
\end{equation*}
\end{definition}

As examples, nonlinear matrix functions $-X^{1/2}$ and $-\ln X$ are ${\cal S}^{n\times n}_+$-convex; see e.g.~\cite{HJ91}.

Throughout this paper, we make the following standard assumptions on (VIP):
\begin{assumption}
\label{assumpA} {\ }

\noindent{\rm(H$_1$)} The solution set $\setS_{\mbox{\tiny VIP}}$ of (VIP) is nonempty.

\noindent{\rm(H$_2$)} $G$ is $L$-Lipschitz on $\UU$.

\noindent{\rm(H$_3$)} $J$ is a proper convex low-semicontinuous (l.s.c.) function (not necessarily differentiable).

\noindent{\rm(H$_4$)} $\Theta$ is $\mathbf{C}$-convex $\RR^n\rightarrow\RR^m$.

\noindent{\rm(H$_5$)} $\Theta(u)$ is Lipschitz with constant $\tau$ on an open subset $\mathcal{O}$ containing $\mathbf{U}$, where
\begin{equation*}\label{Theta_Lipschitz}
\|\Theta(u)-\Theta(v)\|\leq\tau\|u-v\|, \,\,\, \forall u,v\in\mathcal{O}.
\end{equation*}

\noindent{\rm(H$_6$)} The following constraint qualification condition holds:
If $\interior (\mathbf{C}) \not= \emptyset$, then
%When $\mathring{\mathbf{C}}\neq\emptyset$, we have
\begin{equation}\label{Theta_CQC_neq}
%\mbox{{\bf CQC:}}\qquad
\Theta(\mathbf{U})\cap \interior (\mathbf{C}) %(-\mathring{\mathbf{C}})
\neq\emptyset;
%\qquad\qquad\qquad\qquad\qquad\quad
\end{equation}
if $\mathbf{C}=\{0\}$, then
$0\in \interior (\Theta(\mathbf{U}))$.
\end{assumption}

%\end{assumptionA}

\subsection{ (VIP) and its primal-dual Lagrangian variational systems }\label{sec:Lagrangian}
%The starting points is the well known observation that
It is well-known that if
$u^*$ is a solution of (VIP) then $u^*$ is the solution of the following convex optimization problem
\begin{equation}\label{Prob:OP}
\begin{array}{lcl}
\mbox{OPT($u^*$)}  &\min\limits_{u\in \mathbf{U}}       & \langle G(u^*),u\rangle+J(u) \\
                    &\rm {s.t}  & \Theta(u)\in -\mathbf{C},
\end{array}
\end{equation}
and {\em vice versa}. The standard Lagrangian function for $\mbox{\rm OPT}(u^*)$ is
\begin{equation}\label{func:L_VIP}
L_{u^*}(u,p):=\langle G(u^*),u\rangle+J(u)+\langle p,\Theta(u)\rangle,
\end{equation}
where $p$ is an optimal Lagrangian multiplier associated with $\mbox{OPT}(u^*)$. By the standard saddle point optimality condition (cf.~\cite{Zhu03}) we know that $(\hat{u},\hat{p})\in\mathbf{U}\times\mathbf{C}^*$ is a saddle point of $L_{u^*}$ if the following inequalities hold
\begin{equation}\label{eq:L}
L_{u^*}(\hat{u},p)\leq L_{u^*}(\hat{u},\hat{p})\leq L_{u^*}(u,\hat{p}),\quad\forall (u,p)\in\mathbf{U}\times\mathbf{C}^*.
\end{equation}
Following~\cite{Zhu03}, we introduce an augmented Lagrangian function for $\mbox{\rm OPT}(u^*)$:
\begin{equation}\label{eq:AL}
L_{u^*}^{\gamma}(u,p) := \langle G(u^*),u\rangle+J(u)+\varphi(\Theta(u),p)
\end{equation}
with
\begin{equation}\label{eq:PHI}
\varphi (\theta ,p) := \max_{q\in\mathbf{C}^*} \,\, \langle q, \theta \rangle-\|q-p\|^2/2\gamma,
\end{equation}
where $\gamma>0$ is a fixed parameter. The functions $\varphi (\theta , p)$ and $L_{u^*}^{\gamma}(u, p)$ have some
useful properties as stated in the following two propositions, whose proofs can be found in~\cite{CohenZ}.
\begin{proposition}[Properties of $\varphi$]\label{THM1_VIP}
Consider the function $\varphi (\theta , p)$ as defined by \eqref{eq:PHI}.
Then,
\begin{itemize}
\item [{\rm (i)}] $\varphi$ is convex in $\theta$ and concave in $p$;
\item [{\rm (ii)}] $\varphi$ is differentiable in $\theta$ and $p$, and one
has
\begin{eqnarray}
&&\nabla_\theta \varphi(\theta , p)=\Pi (p+\gamma\theta ),\label{PI1_VIP}\\
&&\nabla_p \varphi (\theta , p)=[\Pi (p+\gamma\theta )-p]/\gamma, \label{PI2_VIP}\\
%&& \sout{{\color{red} \varphi (\theta , p)=[\|\Pi (p+\gamma\theta )\|^2-\|p\|^2]/(2\gamma),}} \nonumber \\
&& \varphi (\theta , p)=\frac{1}{2\gamma}\left[ \|\Pi (p+\gamma\theta )\|^2-\|p\|^2\right], %{\color{brown}\mbox{(We agree.)}}}
\label{PI3_VIP}
\end{eqnarray}
where the operator $\Pi(x)$ stands for the projection of $x$ onto $\mathbf{C}^*$;
\item [{\rm (iii)}] For any $p\in\mathbf{C}^*$, the function $\langle p,\Theta(u)\rangle$ and function $\varphi(\Theta(u),p)$ are continuous and convex in $u\in\mathbf{U}$.
\end{itemize}
\end{proposition}

\begin{proposition}[Relationship between $L_{u^*}$ and $L_{u^*}^{\gamma}$]\label{THM2_VIP}
Consider the Lagrangian function $L_{u^*}(u,p)$ and the augmented Lagrangian function $L_{u^*}^{\gamma}(u, p)$ defined according to~\eqref{eq:L} and~\eqref{eq:AL} respectively. Suppose Assumption~\ref{assumpA} holds.  Then,
%\begin{itemize}
%\item [{\rm(i)}]
$L_{u^*}$ and $L_{u^*}^{\gamma}$ have the same sets of saddle points $\hat\UU\times\hat{\mathbf{P}}$ on $\UU\times\mathbf{C}^*$ and $\UU\times\RR^m$, respectively.
%\item [{\rm(ii)}]
Specifically, for any solution $u^*$ of (VIP), there is $p^*\in\mathbf{C}$ such that $(u^*,p^*)$ is a saddle point of $\mbox{\rm OPT}(u^*)$, satisfying
\begin{eqnarray}\label{eq:L_saddlepoint}
L_{u^*}(u^{*},p)\leq L_{u^*}(u^{*},p^*)\leq L_{u^*}(u, p^* ),\quad\forall u\in\UU,\; \forall p\in\mathbf{C}^*,
\end{eqnarray}
and
\begin{eqnarray}\label{eq:L_gamma_saddlepoint}
L_{u^*}^{\gamma}(u^{*} ,p)\leq L_{u^*}^{\gamma}(u^{*} ,p^* )\leq L_{u^*}^{\gamma}(u, p^* ),\quad\forall u\in\UU,\; \forall p\in\mathbf{C}^*.
\end{eqnarray}
%\end{itemize}
\end{proposition}

Based on Proposition~\ref{THM2_VIP}, the saddle point $(u^*,p^*)$ of \eqref{eq:L_saddlepoint} and \eqref{eq:L_gamma_saddlepoint} can be characterized respectively by the following Lagrangian-like variational inequality system
\begin{eqnarray}
\mbox{(VIS) }  &&\langle G(u^*),u-u^*\rangle+J(u)-J(u^*) +\langle p^*,
\Theta (u)-\Theta (u^*)\rangle \geq 0,\, \forall u\in\UU\label{VIS1_VIP} \\
%&&\qquad\qquad\qquad\qquad\qquad\qquad\qquad\qquad\qquad\qquad\qquad \forall u\in\UU,\nonumber\\
&&\langle \Theta (u^*), p-p^*\rangle \leq 0, \quad \forall p\in\mathbf{C}^*, \label {VIS2_VIP}
\end{eqnarray}
and the augmented Lagrangian-like variational inequality system
\begin{eqnarray}
\mbox{(aVIS)} &&\langle G(u^*),u-u^*\rangle+J(u)-J(u^*) +\varphi (\Theta (u),
p^*) -\varphi (\Theta (u^*),p^*) \geq 0, \;\forall u\in\UU  \label{VIS3_VIP}\\
%&&\qquad\qquad\qquad\qquad\qquad\qquad\quad -\varphi (\Theta (u^*),p^*) \geq 0, \;\forall u\in\UU, \nonumber\\
&&\langle \Pi (p^* +\gamma\Theta (u^*)) -p^*, p-p^*\rangle \le 0, \quad
\forall p\in\mathbf{C}^*.\label {VIS4_VIP}
\end{eqnarray}
By the convexity of $J(u)$ and Lemma 3.3 in~\cite{Zhu03}, we further obtain the following equivalent formulation of (aVIS), which we shall call (AVIS) below:
\begin{eqnarray}
\mbox{(AVIS)}&&\langle G(u^*),u-u^*\rangle+J(u)-J(u^*) +\langle \Pi (p^* +\gamma\Theta
(u^*)), \Theta (u)-\Theta (u^*)\rangle \geq 0,\, \forall u\in\UU  \label{VIS5_VIP}\\
%&&\qquad\qquad\quad+\langle \Pi (p^* +\gamma\Theta
%(u^*)), \Theta (u)-\Theta (u^*)\rangle \geq 0, \forall u\in\UU,\nonumber\\
&&\langle \Pi (p^* +\gamma \Theta (u^* )) -p^* , p-p^*\rangle \leq 0, \,\, \forall p\in\mathbf{C}^*.\label {VIS6_VIP}
\end{eqnarray}
Up till this point, we have shown that for any $u^*$ of (VIP) there is $p^*\in\mathbf{C}^*$ such that $(u^*,p^*)$ solves both (VIS) and (AVIS). Next proposition shows that any solution of (VIS)/(AVIS) is in fact a solution for (VIP) as well.

\begin{proposition}[Solutions of (VIP) and (VIS)/(AVIS)]\label{THM3_VIP}
Two variational inequality systems (VIS) and (AVIS) have the same nonempty solution set $\setS_{\mbox{\tiny VIS}}$. For any $(u^*,p^*)\in\setS_{\mbox{\tiny VIS}}$, $u^*$ solves (VIP).
\end{proposition}

\begin{proof}
The proof is similar to that of Theorem 3.3 in~\cite{Zhu03}. We omit the details for succinctness.
\end{proof}

The conclusions in Proposition~\ref{THM3_VIP} encourage us to directly develop a primal-dual method for solving (VIP) by means of solving (VIS) and (AVIS). Finally, we observe that the solution $(u,p)$ of (VIS) satisfies the following KKT condition of (VIP):
\begin{equation}\label{eq:KKT}
\left\{
\begin{array}{l}
0\in G(u)+\partial J(u) +\left(\partial\Theta(u)\right)^{\top}p+\mathcal{N}_{\mathbf{U}}(u)      \\
0\in-\Theta(u)+\mathcal{N}_{\mathbf{C}^*}(p),
\end{array}
\right.
\end{equation}
where %$\mathcal{N}_{\mathbf{U}}(u):=\{\xi \mid \langle\xi,\zeta-u\rangle\leq0, \forall\zeta\in\mathbf{U}\}$ is the normal cone at $u$ with respect to a given closed convex set $\mathbf{U}$.
$\mathcal{N}_{\mathbf{S}}(x):=\{\xi \mid \langle\xi,\zeta-x\rangle\leq0, \forall\zeta\in\mathbf{S}\}$ is the normal cone at $x$ with respect to a given closed convex set $\mathbf{S}$.
For this reason, the solution $(u,p)$ of (VIS) can be understood as a KKT point of (VIP).

\subsection{A (mixed) Minty variational inequality system and the primal-dual variational coherence condition}\label{subsec:minty}
The (mixed) Minty variational inequality system with respect to (VIS) is to find $(u^\natural , p^\natural )\in \UU\times\mathbf{C}^*$ such that
\begin{eqnarray}
\mbox{(MVIS)} &&\langle G(u ),u-u^\natural \rangle+J(u)-J(u^\natural ) +\langle p^\natural ,
\Theta (u)-\Theta (u^\natural )\rangle \geq 0,\, \forall u\in\UU \label{MVIS1_VIP}\\
%&&\qquad\qquad\qquad\qquad\qquad\qquad\qquad\qquad\qquad\qquad\qquad \forall u\in\UU,\nonumber\\
&&\langle  \Theta (u^\natural ), p-p^\natural \rangle \leq 0, \quad \forall p\in
\mathbf{C}^*, \label {MVIS2_VIP}
\end{eqnarray}
in which case $(u^\natural ,p^\natural )$ is referred to as a weak primal-dual solution of (VIP). Let $\setS_{\mbox{\tiny MVIS}}$ be the set of solutions for (MVIS).  If $G$ is not monotone, then we can only ensure $\setS_{\mbox{\tiny MVIS}} \subseteq \setS_{\mbox{\tiny VIS}}$.
%, however solution of MVIS may not exist even when VIS solution exist.
Moreover, it can happen that $\setS_{\mbox{\tiny MVIS}}=\emptyset$ while $\setS_{\mbox{\tiny VIS}}\not=\emptyset$. However,
if $G$ is monotone, $\Theta$ is $\CC$-convex and $J$ is convex, then $\emptyset \not= \setS_{\mbox{\tiny MVIS}} = \setS_{\mbox{\tiny VIS}}$. %the two solution sets are equivalent.
For ease of referencing, let us introduce the following notion of {\em primal-dual variational coherence condition}:

\begin{definition}[Primal-dual variational coherence condition]\label{def:pd-coherent}
The variational inequality system (VIS) is said to satisfy the primal-dual variational coherence condition (or is called primal-dual variational coherent henceforth) if and only if $\setS_{\mbox{\tiny MVIS}}\neq\emptyset$.
\end{definition}

\begin{proposition}[Solution of (MVIS)]\label{Minty}
Let $(u^\natural ,p^\natural )$ be the solution of {\rm (MVIS)}. Then,
\begin{itemize}
\item[{\rm(i)}] $u^\natural$ solves {\rm (VIP)};
\item[{\rm(ii)}] $(u^\natural ,p^\natural )$ solves {\rm (VIS)}.
\end{itemize}
\end{proposition}
\begin{proof}$\quad$
\begin{itemize}
\item[{\rm(i)}] For all $u\in \Gamma$, we have $\Theta (u)\in -\CC$. Thus $\langle p^\natural , \Theta (u)\rangle\leq 0$, and $\langle p^\natural , \Theta (u^\natural )\rangle= 0$. Hence, we have
\begin{eqnarray}\label{eq:MVIS-VIS-1}
\langle G(u), u-u^\natural \rangle +J(u)-J(u^\natural )\geq -\langle p^\natural , \Theta (u)\rangle\geq 0,\quad\;\forall u\in \Gamma.
\end{eqnarray}
%From~\eqref{MVIS1_VIP} and
By~\eqref{MVIS2_VIP}, we have $ \Theta (u^\natural ) \in -\CC$, and so $u^{\natural}\in\Gamma$. Take any $v\in\Gamma$. Let $u=(1-t)v+tu^{\natural}\in\Gamma$ with $0\leq t<1$. By~\eqref{eq:MVIS-VIS-1} and the convexity of $J$, we have
\begin{eqnarray*}
&&\langle G((1-t)v+tu^{\natural}),(1-t)(v-u^{\natural})\rangle+(1-t)\left[J(v)-J(u^{\natural})\right]\\
&\geq&\langle G((1-t)v+tu^{\natural}),(1-t)(v-u^{\natural})\rangle+J((1-t)v+tu^{\natural})-J(u^{\natural})\\
&\geq&0,
\end{eqnarray*}
and so, $\langle G((1-t)v+tu^{\natural}),v-u^{\natural}\rangle+J(v)-J(u^{\natural})\geq0$.
Letting $t\uparrow 1$, by the continuity of $G$, we have
\[
\langle G(u^{\natural}),v-u^{\natural}\rangle+J(v)-J(u^{\natural})\geq0,\quad\forall v\in\Gamma,
\]
which shows that $u^\natural$ solves (VIP).
\item[{\rm(ii)}] Observe that $\langle p^{\natural},\Theta(u) \rangle$ is convex in $u$, by the same argument as (i), we have
\[
\langle G(u^{\natural}),v-u^{\natural}\rangle+J(v)-J(u^{\natural})+\langle p^{\natural},\Theta(v)-\Theta(u^{\natural})\rangle\geq0,\quad\forall v\in\UU.
\]
Therefore, $(u^{\natural},p^{\natural})$ solves (VIS).
\end{itemize}
\end{proof}
%\begin{remark}
%By~\ref{Minty}, in the analysis of this paper, we can view solution $(u^*,p^*)\in \setS_{\mbox{\tiny VIS}}$ is same as one solution as $(u^\natural ,p^\natural )\in \setS_{\mbox{\tiny MVIS}}$, whenever (VIS) satisfies primal-dual variational coherence condition.
%\end{remark}

Proposition~\ref{Minty} assures that if (VIS) satisfies the primal-dual variational coherence condition then it suffices to find $(u^\natural ,p^\natural )\in \setS_{\mbox{\tiny MVIS}}$, because in this case $u^\natural\in \setS_{\mbox{\tiny VIP}}$ and $(u^\natural,p^\natural)\in \setS_{\mbox{\tiny VIS}}$. It remains to find more tangible conditions that can ensure (VIS) to be primal-dual variational coherent, which motivates the discussion in the next section.

\section{Sufficient Conditions for the Primal-Dual Variational Coherence }\label{sec:monotone}

It turns out that the following generalized monotonicity conditions will lead to the desired primal-dual variational coherence condition stipulated in Definition~\ref{def:pd-coherent}. Note that these are merely known classes of sufficient conditions to guarantee the property required in Definition~\ref{def:pd-coherent}, and they are by no means exhaustive. In Section~\ref{sec:ALAVI_convergence}, we will present our bid to solve (VIP)---a new algorithm to be called
Augmented Lagrangian Approach to Variational Inequalities
%\sout{Modified Augmented Lagrangian Decomposition}}
(ALAVI). The convergence of ALAVI does not require (VIP) to be monotone; however, it does require (VIS) to be primal-dual variational coherent.

\begin{definition}[Generalized monotonicity of mapping $G$]\label{def:general_monotone}\quad\\
{\rm 1.} {\bf (Monotone)} The mapping $G$ is monotone on $\UU$ if for all $u,v\in\UU$, we have
\[
\langle G(u)-G(v),u-v\rangle\geq0.
\]
{\rm 2.} {\bf (Co-coercive)} We call the mapping $G$ is co-coercivity on $\UU$ if for all $u,v\in\UU$ such that
\[
\langle G(u)-G(v),u-v\rangle \geq \mu \|G(u)-G(v)\|^2.
\]
{\rm 3.} {\bf (Star-monotone)} Let $u^*$ be one solution of (VIP). The mapping $G$ is star-monotone at $u^*$ on $\UU$, if for all $u\in\UU$ we have
\[
\langle G(u)-G(u^*),u-u^*\rangle\geq0.
\]
{\rm 4.} {\bf (Pseudo-monotone)} The mapping $G$ is pseudo-monotone on $\UU$, if for all $u,v\in\UU$ such that
\[
\langle G(u),v-u\rangle\geq 0 \Longrightarrow\langle G(v),v-u\rangle\geq0.
\]
{\rm 5.} {\bf ($J$-Pseudo-monotone)} The mapping $G$ is $J$-pseudo-monotone on $\UU$, if for all $u,v\in\UU$ such that
\[
\langle G(u),v-u\rangle+J(v)-J(u)\geq 0 \Longrightarrow\langle G(v),v-u\rangle+J(v)-J(u)\geq 0,
\]
where $J$ is a function on $\UU$.

\noindent{\rm 6.} {\bf (($J+\langle p^*,\Theta \rangle )$-Pseudo-monotone)} The mapping $G$ is $(J+\langle p^*,\Theta \rangle )$-pseudo-monotone on $\UU$, if for multiplier $p^*$ of (VIP) and all $u,v\in\UU$ such that
\[
\begin{aligned}
\langle G(u),v-u\rangle&+J(v)-J(u) +\langle p^*,\Theta (v)-\Theta (u) \rangle\geq0 \\
&\Longrightarrow\langle G(v),v-u\rangle+J(v)-J(u) +\langle p^*,\Theta (v)-\Theta (u) \rangle\geq 0,
\end{aligned}
\]
where $J$ is a function on $\UU$.

\noindent{\rm 7.} {\bf (Quasi-monotone)} The mapping $G$ is quasi-monotone on $\UU$, if for all $u,v\in\UU$ such that
\[
\langle G(u),v-u\rangle > 0 \Longrightarrow\langle G(v),v-u\rangle \geq 0.
\]

\noindent{\rm 8.} {\bf ($J$-Quasi-monotone)} The mapping $G$ is $J$-quasi-monotone on $\UU$, if for all $u,v\in\UU$ such that
\[
\langle G(u),v-u\rangle+J(v)-J(u)> 0 \Longrightarrow\langle G(v),v-u\rangle+J(v)-J(u)\geq 0,
\]
where $J$ is a function on $\UU$.

\noindent{\rm 9.} {\bf (($J+\langle p^*,\Theta \rangle )$-Quasi-monotone)} The mapping $G$ is $(J+\langle p^*,\Theta \rangle )$-quasi-monotone on $\UU$, if for multiplier $p^*$ of (VIP) and all $u,v\in\UU$ such that
\[
\begin{aligned}
\langle G(u),v-u\rangle&+J(v)-J(u) +\langle p^*,\Theta (v)-\Theta (u) \rangle > 0\\
&\Longrightarrow\langle G(v),v-u\rangle+J(v)-J(u) +\langle p^*,\Theta (v)-\Theta (u) \rangle\geq 0,
\end{aligned}
\]
where $J$ is a function on $\UU$.
\end{definition}

Examples are constructed to illustrate that these classes are non-identical. %, which shows the following.
In order not to disrupt the main flow of the presentation, these examples are delegated to Appendix~\ref{appendix}.

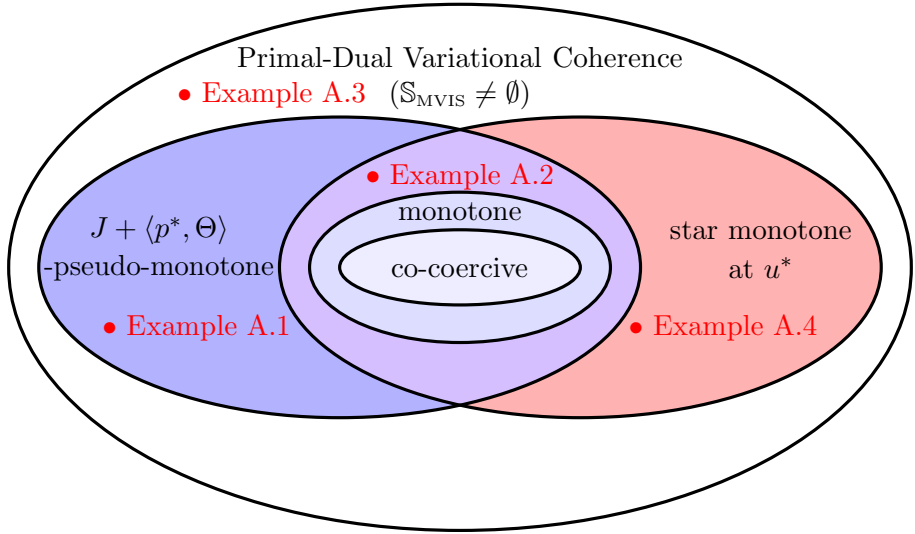
\begin{figure}[h]
\centering
\begin{center}
\begin{tikzpicture}
[%?????????latex ?????
>=latex,
%???????????
node distance=5mm,
% hv path ???????????????????????????????vh ????skip loop ???
%???-??-??? vskip loop ?????-???-??
 ract/.style={draw=blue!50, fill=blue!5,rectangle,minimum size=6mm, very thick, font=\itshape, align=center},
 racc/.style={rectangle, align=center},
 ractm/.style={draw=red!100, fill=red!5,rectangle,minimum size=6mm, very thick, font=\itshape, align=center},
 cirl/.style={draw, fill=yellow!20,circle,   minimum size=6mm, very thick, font=\itshape, align=center},
 raco/.style={draw=green!500, fill=green!5,rectangle,rounded corners=2mm,  minimum size=6mm, very thick, font=\itshape, align=center},
 hv path/.style={to path={-| (\tikztotarget)}},
 vh path/.style={to path={|- (\tikztotarget)}},
 skip loop/.style={to path={-- ++(0,#1) -| (\tikztotarget)}},
 vskip loop/.style={to path={-- ++(#1,0) |- (\tikztotarget)}}]
\begin{scope}[blend group = soft light]
%\fill[red!10!white]   (0:0) ellipse (6 and 3.5);
\fill[blue!30!white] (0:-1.6) ellipse (4 and 2);
\fill[red!30!white] (0:1.6) ellipse (4 and 2);
\fill[green!30!white] (0:0) ellipse (2 and 1);
\fill[white]  (0:0) ellipse (1.6 and 0.5);
%\fill[yellow!30!white]  (0:1.2) ellipse (0.7 and 0.5);
\end{scope}
\draw[draw=black!100,very thick] (0:0) ellipse (6 and 3.5);
\draw[draw=black!100,very thick] (0:-1.6) ellipse (4 and 2);
\draw[draw=black!100,very thick] (0:1.6) ellipse (4 and 2);
\draw[draw=black!100,very thick] (0:0) ellipse (2 and 1);
\draw[draw=black!100,very thick] (0:0) ellipse (1.6 and 0.5);
%\draw[draw=black!100,very thick] (0:1.2) ellipse (0.7 and 0.5);
\node (b3) [font=\Large] at ( 0,2.8)    {\normalsize Primal-Dual Variational Coherence};
\node (b33) [font=\Large] at ( 0,2.3)    {\normalsize ($\setS_{\mbox{\tiny MVIS}}\neq\emptyset$)};
\node (b333) [font=\Large] at ( -2.5,2.3)    {\normalsize {\color{red}$\bullet$ Example \ref{exap3}}};
\node (b3') [font=\Large] at ( -4,0.5)    {\normalsize $J+\langle p^*,\Theta \rangle$};
\node (b3'') [font=\Large] at ( -4,0)   {\normalsize-pseudo-monotone};
\node (b333) [font=\Large] at ( -3.5,-0.8)    {\normalsize {\color{red}$\bullet$ Example \ref{exap1}}};
%\node (b35) at ( -1.4,-2.4) {\small $\bigoplus$ quadratic growth of $\varphi$};
%\node (b35p) at ( -1.4,-2.65) {\small and $\overline{\setX}$ is bounded};
%\node (b36) at ( -1.4,-3.2) {\small $\bigoplus$ metric subregularity of $\partial\varphi$};
%\node (b36p) at ( -1.4,-3.45) {\small and $\overline{\setX}$ is bounded};
\node (b4) [font=\Large] at ( 0,0.75)    {\normalsize monotone};
\node (b4) [font=\Large] at ( 0,1.2)    {\normalsize {\color{red}$\bullet$ Example \ref{exap2}}};
\node (b5) [font=\Large] at ( 0,0)    {\normalsize co-coercive};
\node (b6) [font=\Large] at ( 4,0.5)    {\normalsize star monotone};
\node (b6') [font=\Large] at ( 4,0)    {\normalsize at $u^*$};
\node (b333) [font=\Large] at ( 3.5,-0.8)    {\normalsize {\color{red}$\bullet$ Example \ref{exap4}}};
%\path (b3) edge[-] (-2.9,-3);
%\path (b6) edge[->] (b5);
%\path (b5) edge[->] (b4);
%\path (b4) edge[->] (b3);
%\path (-2.9,-2.2) edge[-] (-0.1,-2.2);
%\path (-2.9,-3) edge[-] (1.3,-3);
%\path (1.3,-3) edge[->] (b6);
%\path (-0.1,-2.2) edge[->] (b5);
\end{tikzpicture}
\caption{Relationship between generalized monotonicity and primal-dual variational coherence}\label{fig:1}
\end{center}
\end{figure}
Figure~\ref{fig:1} and the following four examples illustrate the relationship between generalized monotonicity and the primal-dual variational coherent condition as given %respectively
in Definition~\ref{def:general_monotone} and Definition~\ref{def:pd-coherent}.
%{\color{red} (Do we need the black colored ellipse in the middle of Figure~\ref{fig:1}? {\color{brown} We have changed the color of Figure~\ref{fig:1}. Please Check.})}

The next proposition provides some sufficient conditions for $\setS_{\mbox{\tiny MVIS}} \neq \emptyset$ (or according to Definition~\ref{def:pd-coherent}, (VIS) is primal-dual variational coherent), which in turn will guarantee the convergence of our algorithm ALAVI to be proposed in Section~\ref{sec:ALAVI_convergence}.
%which guarantees the convergence of algorithm ALAVI to solve (VIP) in Section~\ref{sec:monotone}.

\begin{proposition}[Sufficient conditions for $\setS_{\mbox{\tiny MVIS}}\neq\emptyset$]
Suppose $(u^*, p^*)\in \setS_{\mbox{\tiny VIS}}$. If one of the following conditions holds, then (VIS) satisfies the primal-dual variational coherence condition (Definition~\ref{def:pd-coherent}), namely $\setS_{\mbox{\tiny MVIS}}\neq\emptyset$.
\begin{itemize}
\item [{\rm(i)}] $G$ is star monotone at $u^*$ on $\UU$.
\item [{\rm(ii)}] $G$ is $(J+\langle p^*,\Theta \rangle )$-pseudo-monotone on $\UU$.
\item [{\rm(iii)}] Let $\setS_{\mbox{\tiny T}}=\{ (u^*,p^*)\in \setS_{\mbox{\tiny VIS}}\mid\langle G(u^*), u-u^*\rangle +J(u)-J(u^*)+\langle p^*, \Theta (u)-\Theta (u^*)\rangle =0,\; \forall u\in \UU\}$ and $\setS_{\mbox{\tiny N}}=\setS_{\mbox{\tiny VIS}}\setminus\setS_{\mbox{\tiny T}}$. $G$ is $(J+\langle p^*,\Theta \rangle)$-quasi-monotone on $\UU$ and $\setS_{\mbox{\tiny N}}\neq \emptyset$.
\item [{\rm(iv)}] In the special case when $J(u)=0,\;p^*=0$, any one of the following holds:
\begin{itemize}
\item [{\rm(a)}] $G$ is the gradient of a differentiable quasi convex function;
\item [{\rm(b)}] $G$ is quasi monotone, $G\neq 0$ on $\UU$ and $\UU$ is bounded;
\item [{\rm(c)}] $G$ is quasi monotone, $G\neq 0$ on $\UU$ and there exists a positive number $r$ such that, for every $u\in \UU$ with $\|u\|\geq r$, there exists $v\in \UU$ such that $\|v\|\leq r$ and $\langle G(u), v-u\rangle \leq 0$;
\item [{\rm(d)}] $G$ is pseudo monotone at $u^*\in \setS_{\mbox{\tiny VIP}}$ (star pseudo monotone);
\item [{\rm(e)}] $G$ is quasi monotone at $u^*\in \setS_{\mbox{\tiny VIP}}$ (star quasi monotone) and $G(u^*)$ is not normal to $\UU$ at $u^*$.
\end{itemize}
\end{itemize}
\end{proposition}
\begin{proof}{\ }

\begin{itemize}
\item[{\rm(i)}] By the star monotonicity of $G$ at $u^*$, we have $\langle G(u)-G(u^*), u-u^*\rangle \ge 0, \; \forall u\in \UU$, and so,
\begin{eqnarray}\label{p21}
\langle G(u), u-u^*\rangle \ge \langle G(u^*), u-u^*\rangle .
\end{eqnarray}
Adding $J(u)-J(u^*)+\langle p^*, \Theta (u)-\Theta (u^*)\rangle $ on two sides of \eqref {p21}, one gets
\begin{eqnarray}
\langle G(u), u-u^*\rangle +J(u)-J(u^*)+\langle p^*, \Theta (u)-\Theta (u^*)\rangle \ge 0.
\end{eqnarray}
Together with \eqref{VIS2_VIP}, the proof is complete for this case. %we complete the proof for this case.
\item[{\rm(ii)}] The result directly follows from the $(J+\langle p^*,\Theta \rangle )$-pseudo-monotonicity of $G$ on $\UU$ and (\ref {VIS2_VIP}).
\item[{\rm(iii)}] Let $(u^*,p^*)\in \setS_{\mbox{\tiny N}}$. For any fixed $u_1\in \UU$, since $ \setS_{\mbox{\tiny N}}\subseteq \setS_{\mbox{\tiny VIS}}$, we have
\[
\langle G(u^*), u_1-u^*\rangle +J(u_1)-J(u^*)+\langle p^*, \Theta (u_1)-\Theta (u^*)\rangle \ge 0 .
\]
By the convexity of $J(\cdot)$, $\langle p^*, \Theta(\cdot)\rangle$ and the $(J+\langle p^*, \Theta(\cdot)\rangle$-quasi-monotonicity of $G$, and by a similar argument as in Lemma 3.1 of~\cite{HS96}, we conclude that at least one of the following must hold:
\[
\langle G(u_1), u_1-u^*\rangle +J(u_1)-J(u^*)+\langle p^*, \Theta (u_1)-\Theta (u^*)\rangle \ge 0
\]
or
\[
\langle G(u^*), u-u^*\rangle +J(u)-J(u^*)+\langle p^*, \Theta (u)-\Theta (u^*)\rangle \leq 0 , \;\forall u\in \UU.
\]
The second inequality implies $(u^*,p^*)\in \setS_{\mbox{\tiny T}}$ which contradicts with $(u^*,p^*)\in \setS_{\mbox{\tiny N}}$. Therefore, the first inequality must hold. Because $u_1\in \UU$ can be taken arbitrarily, it follows that $\setS_{\mbox{\tiny MVIS}}\neq \emptyset $.
\item[{\rm(iv)}] The proof can be found in~\cite{MZ01}.
\end{itemize}
\end{proof}

\section{Augmented Lagrangian Approach to Variational Inequality Problems  (ALAVI)}\label{sec:ALAVI_convergence}

\subsection{Scheme ALAVI}
\label{subsec:ALAVI}

For (VIP) (Model~\eqref{Prob:VIP}), in its equivalently reformulated model (AVIS)~\eqref{VIS5_VIP}-\eqref{VIS6_VIP}, one may refer the variable $u$ as {\it primal}, and $p$ as {\it dual}. The crux in our newly proposed algorithm is to introduce two extrapolated sequences: $\{v^k\}$ (associated with $\{u^k\}$) and $\{ q^k\}$ (associated with the partial differential of the augmented Lagrangian term $\varphi(\theta,p)$ with respect to $\theta$~\eqref{PI1_VIP} at $(u^k,p^k)$). The procedure somehow resembles, in spirit, to Nesterov's gradient acceleration algorithm~\cite{N03} and the extra-gradient algorithm for VI~\cite{HZ21}, although the context here is completely different. For one, the underlying VI problem is non-monotone; for two, the procedure involves primal and dual variables. The augmented Lagrangian plays an important role in this design, hence the name ALAVI (Augmented Lagrangian Approach to (VIP)).
%{\color{red} 
%We now introduce an auxiliary variable $v$, and propose an algorithm to solve augmented Lagrangian-like variational inequality system (AVIS)~\eqref {VIS5_VIP}-\eqref{VIS6_VIP}. 
%From Proposition~\ref{THM3_VIP}, this algorithm is called as ALAVI (Augmented Lagrangian Approach to (VIP)).}

\noindent\rule[0.1\baselineskip]{\textwidth}{1.5pt}
{\bf ALAVI (Augmented Lagrangian Approach to (VIP))}

\noindent\rule[0.25\baselineskip]{\textwidth}{0.5pt}

{\bf Input:} $u^1,v^0\in \RR^n$, $p^0\in \RR^m$, and the algorithmic parameters: $0<\eta<1,\, \alpha>0,\, \gamma>0$. Set $k=1$.

{\bf Iterate $k$:}
\begin{eqnarray}
 v^k &:=& (1-\eta)u^k+\eta v^{k-1}; \label{ALAVI_relax}\\
 q^k &:=& \Pi\left(p^k+\gamma\Theta(u^k)\right);  \label{update q vector}\\
 u^{k+1} &:=& \arg\min_{u\in\UU} \,\, \langle G(u^k),u\rangle+J(u)+\langle q^k,\Theta(u)\rangle+\frac{1}{2\alpha}\|u-v^k\|^2; \label{ALAVI_primal}\\
%&&\qquad\qquad\qquad\qquad\qquad\qquad\qquad\qquad\quad\mbox{with}\;q^k=\Pi\left(p^k+\gamma\Theta(u^k)\right);\nonumber\\
 p^{k+1} &:=& \Pi\left(p^k+\gamma\Theta(u^{k+1})\right)\label{ALAVI_dual}; \\
 k &:=& k+1, \mbox{ and return to {\bf Iterate $k$}.} \nonumber
\end{eqnarray}

\noindent\rule[0.25\baselineskip]{\textwidth}{0.5pt}

In the primal updating subproblem~\eqref{ALAVI_primal}, the iteration $v^k$ is taken as the reference point in the regularization term $\frac{1}{2\alpha}\|u-v^k\|^2$ rather than the last iterate $u^k$. The value of $v$ is updated through a convex combination of last iterates $v^{k-1}$ and $u^k$.
Remark that subproblem~\eqref{ALAVI_primal} is specially susceptible to decomposition. For example, if
\begin{equation}\label{eq:decomposition}
\UU=\UU_1\times\UU_2\times \cdots\times\UU_N, \; u_i\in\UU_i\subseteq \RR^{n_i}
\end{equation}
%which means that the implicit constraints defined by $\UU$ are decoupled. From the point of view of decomposition, the interest of ALAVI algorithm is following. Suppose that
and $\Theta(u)$ and $J(u)$ are additive, i.e.,
\[
\Theta(u)=\sum_{i=1}^N\Theta_i(u_i)\quad\mbox{and}\quad J(u)=\sum_{i=1}^NJ_i(u_i),
\]
then, subproblem~\eqref{ALAVI_primal} can be decomposed into $N$ independent subproblems with respect to the decomposition~\eqref{eq:decomposition}:
\begin{equation}
u_{i}^{k+1}=\arg\min_{u_i\in\UU_i}\langle G_i(u^k),u_i\rangle+J_i(u_i)+\langle q^k,\Theta_i(u_i)\rangle+\frac{1}{2\alpha}\|u_i-v_i^k\|^2, \, i=1,2,...,N.
\end{equation}

\subsection{A Lyapunov function and its analysis}\label{subsec:Lyapnov}

Let $u^*$ be a solution of (VIP). For any $(u,p)\in\UU\times\mathbf{C}^*$, we construct the following Lyapnov function
\begin{eqnarray}\label{eq:Lambda_VIP}
\Lambda^k(u,p):=\beta_1\|v^k-u\|^2+\beta_2\|u^{k-1}-u^k\|^2+\beta_3\|p^{k-1}-p\|^2,
\end{eqnarray}
where $\beta_1=\frac{1}{2\alpha(1-\eta)}$, $\beta_2=\frac{\gamma\tau^2+L+\tau}{2}$ and $\beta_3=\frac{1}{2\gamma}$.
\begin{lemma}[Estimation on the change of $\Lambda^k(u,p)$]\label{lemma:bound1_VIP}
Suppose Assumption~\ref{assumpA} holds, and the parameters satisfy $\frac{\sqrt{5}-1}{2}\leq\eta<1$, $0<\gamma<\frac{1}{\tau}$ and $0<\alpha\leq\frac{1}{2(\gamma\tau^2+L+\tau)\eta}$. Let the sequence $\{(v^k,u^k,p^k) \mid k=1,2,...\}$ be generated by ALAVI. Then, for all $(u,p)\in\UU\times\mathbf{C}^*$,
\begin{eqnarray*}
\Lambda^{k}(u,p)-\Lambda^{k+1}(u,p)&\geq&\left[L_{u^*}(u^{k},p)-L_{u^*}(u,q^{k})\right]+\langle G(u^k)-G(u^*),u^k-u\rangle\\
&&+\rho \left[\|w^{k-1}-w^k\|^2+\|p^k-q^k\|^2\right],
\end{eqnarray*}
where $w:=\left(\begin{array}{c} v\\ u\\ p\end{array}\right)$ and $\rho:=\min\left\{\frac{\eta}{2\alpha(1-\eta)^2},\frac{\tau}{2},\frac{1-\tau\gamma}{2\gamma},\frac{1}{\gamma}\right\}>0$.
\end{lemma}

\begin{proof}
Let us prove the lemma in three steps.

\emph{Step 1: Establish a bound on $L_{u^*}(u^{k},q^k)-L_{u^*}(u,q^k)$.}

First, observe that the unique solution $u^{k+1}$ of the primal subproblem~\eqref{ALAVI_primal} can be characterized by the following optimality condition: %variational inequality:
\begin{eqnarray}\label{eq:VI}
\langle G(u^{k}),u-u^{k+1}\rangle+J(u)-J(u^{k+1})+\langle q^k,\Theta(u)-\Theta(u^{k+1})\rangle\\
+\frac{1}{\alpha}\langle u^{k+1}-v^k, u-u^{k+1}\rangle
\geq 0,\forall u\in\mathbf{U}.\nonumber
\end{eqnarray}
Denoting
\[
\Delta^k := \langle G(u^k),u^k-u\rangle+J(u^k)-J(u)+\langle q^k,\Theta(u^k)-\Theta(u)\rangle ,
\]
inequality~\eqref{eq:VI} further leads to
\begin{eqnarray}\label{eq:primal_1}
\Delta^k&\leq&\langle G(u^{k}),u^k-u^{k+1}\rangle+J(u^k)-J(u^{k+1})+\langle q^k,\Theta(u^k)-\Theta(u^{k+1})\rangle\nonumber\\
&&+\frac{1}{\alpha}\langle u^{k+1}-v^k, u-u^{k+1}\rangle\nonumber\\
&=&\langle G(u^{k}),u^k-u^{k+1}\rangle+J(u^k)-J(u^{k+1})+\langle q^k,\Theta(u^k)-\Theta(u^{k+1})\rangle\nonumber\\
&&+\frac{1}{2\alpha}\left[\|u-v^k\|^2-\|u-u^{k+1}\|^2-\|v^k-u^{k+1}\|^2\right].
\end{eqnarray}
By the updating formula %the recursion update for ``$v$''~
\eqref{ALAVI_relax}, we have
\begin{eqnarray}
u^{k+1}&=&\frac{1}{1-\eta}v^{k+1}-\frac{\eta}{1-\eta}v^k;\label{eq:primal_v1}\\
v^{k+1}-v^k&=&(1-\eta)(u^{k+1}-v^k);\label{eq:primal_v2}\\
u^k-v^{k-1}&=&\frac{1}{\eta}(u^k-v^k).\label{eq:primal_v3}
\end{eqnarray}
Furthermore,
\begin{eqnarray}\label{eq:primal_2}
\|u-u^{k+1}\|^2&\overset{\mbox{\tiny \eqref{eq:primal_v1}}}{=}&\left\|\frac{1}{1-\eta}(u-v^{k+1})-\frac{\eta}{1-\eta}(u-v^k)\right\|^2\nonumber\\
&=&\left(\frac{1}{1-\eta}\right)^2\|u-v^{k+1}\|^2+\left(\frac{\eta}{1-\eta}\right)^2\|u-v^k\|^2\nonumber\\
&&-\frac{2\eta}{(1-\eta)^2}\langle u-v^{k+1},u-v^k\rangle\nonumber\\
&=&\left[\left(\frac{1}{1-\eta}\right)^2-\frac{\eta}{(1-\eta)^2}\right]\|u-v^{k+1}\|^2 %\nonumber\\
    + \left[\left(\frac{\eta}{1-\eta}\right)^2-\frac{\eta}{(1-\eta)^2}\right]\|u-v^k\|^2 \nonumber\\
&&+\frac{\eta}{(1-\eta)^2}\left[\|u-v^k\|^2+\|u-v^{k+1}\|^2-2\langle u-v^{k+1},u-v^k\rangle\right]\nonumber\\
&=&\frac{1}{1-\eta}\|u-v^{k+1}\|^2-\frac{\eta}{1-\eta}\|u-v^k\|^2+\frac{\eta}{(1-\eta)^2}\|v^k-v^{k+1}\|^2\nonumber\\
&\overset{\mbox{\tiny \eqref{eq:primal_v2}}}{=}&\frac{1}{1-\eta}\|u-v^{k+1}\|^2-\frac{\eta}{1-\eta}\|u-v^k\|^2+\eta\|v^k-u^{k+1}\|^2.
\end{eqnarray}
Substituting~\eqref{eq:primal_2} into~\eqref{eq:primal_1} yields a bound on $\Delta^k$:
\begin{eqnarray}\label{eq:primal_3}
\Delta^k&\leq&\langle G(u^{k}),u^k-u^{k+1}\rangle+J(u^k)-J(u^{k+1})+\langle q^k,\Theta(u^k)-\Theta(u^{k+1})\rangle\nonumber\\
&&+\frac{1}{2\alpha}\left[\frac{1}{1-\eta}\|u-v^k\|^2-\frac{1}{1-\eta}\|u-v^{k+1}\|^2-(\eta+1)\|v^k-u^{k+1}\|^2\right].
\end{eqnarray}
%In order to get a better bound for $\Delta^k$, we
To proceed further, let us consider inequality~\eqref{eq:VI} satisfied at iteration $k-1$:
\begin{eqnarray*}
\langle G(u^{k-1}),u-u^{k}\rangle+J(u)-J(u^{k})+\langle q^{k-1},\Theta(u)-\Theta(u^{k})\rangle\\+\frac{1}{\alpha}\langle u^{k}-v^{k-1}, u-u^{k}\rangle
\geq 0, \forall u\in\mathbf{U}.
\end{eqnarray*}
Taking $u=u^{k+1}$ and using~\eqref{eq:primal_v3} in the above inequality, we obtain
\begin{eqnarray}\label{eq:primal_2_2}
& & \langle G(u^{k-1}),u^{k+1}-u^{k}\rangle+J(u^{k+1})-J(u^{k}) \nonumber\\
& & +\langle q^{k-1},\Theta(u^{k+1})-\Theta(u^{k})\rangle
+\frac{1}{\alpha\eta}\langle u^{k}-v^{k}, u^{k+1}-u^{k}\rangle\geq 0.
\end{eqnarray}
Now, observe that the last term of~\eqref{eq:primal_2_2} can be rewritten as
\[
\frac{1}{\alpha\eta}\langle u^{k}-v^{k}, u^{k+1}-u^{k}\rangle=\frac{1}{2\alpha\eta}\left[\|v^k-u^{k+1}\|^2-\|u^k-v^k\|^2-\|u^k-u^{k+1}\|^2\right],
\]
and so \eqref{eq:primal_2_2} gives us
\begin{eqnarray}\label{eq:primal_4}
\langle G(u^{k-1}),u^{k+1}-u^{k}\rangle+J(u^{k+1})-J(u^{k})+\langle q^{k-1},\Theta(u^{k+1})-\Theta(u^{k})\rangle\nonumber\\+\frac{1}{2\alpha\eta}\left[\|v^k-u^{k+1}\|^2-\|u^k-v^k\|^2-\|u^k-u^{k+1}\|^2\right]\geq 0.
\end{eqnarray}
Summing up~\eqref{eq:primal_3} and~\eqref{eq:primal_4}, we obtain
\begin{eqnarray}\label{eq:primal_5}
\Delta^k&\leq&\langle G(u^{k})- G(u^{k-1}),u^k-u^{k+1}\rangle+\langle q^k-q^{k-1},\Theta(u^k)-\Theta(u^{k+1})\rangle\nonumber\\
&&+\frac{1}{2\alpha}\bigg{[}\frac{1}{1-\eta}\|u-v^k\|^2-\frac{1}{1-\eta}\|u-v^{k+1}\|^2-\frac{1}{\eta}\|u^k-v^k\|^2\nonumber\\
&&-\frac{1}{\eta}\|u^k-u^{k+1}\|^2-\left(\frac{\eta^2+\eta-1}{\eta}\right)\|v^k-u^{k+1}\|^2\bigg{]}.
\end{eqnarray}
Now, by the $L$-Lipschitzian property of $G$ as in Assumption~\ref{assumpA}, %the first term in~\eqref{eq:primal_5} has an upper bound
\begin{eqnarray}\label{eq:primal_6}
\langle G(u^{k})- G(u^{k-1}),u^k-u^{k+1}\rangle&\leq& L\|u^{k-1}-u^k\|\cdot\|u^k-u^{k+1}\|\nonumber\\
&\leq&\frac{L}{2}\left[\|u^{k-1}-u^{k}\|^2+\|u^k-u^{k+1}\|^2\right].
\end{eqnarray}
Next, by the nonexpansiveness of projection $\Pi_{\mathcal{S}}(x)$ on a convex set $\mathcal{S}$ (namely $\|\Pi_{\mathcal{S}}(x)-\Pi_{\mathcal{S}}(y)\|\leq\|x-y\|,\,\forall x,y$; see e.g.~\cite{Projectiononconvexsets}) and the fact that $\Theta$ is $\tau$-Lipschitz according to Assumption~\ref{assumpA}, we have %the second term of~\eqref{eq:primal_5} have the following upper bound
\begin{eqnarray}\label{eq:primal_7}
&&\langle q^k-q^{k-1},\Theta(u^k)-\Theta(u^{k+1})\rangle\nonumber\\
&\leq&\|q^k-q^{k-1}\|\cdot\tau\|u^k-u^{k+1}\|\nonumber\\
&\leq&\left(\|p^{k-1}-p^k\|+\gamma\tau\|u^{k-1}-u^{k}\|\right)\cdot\tau\|u^k-u^{k+1}\|\nonumber\\
&\leq&\frac{\tau}{2}\|p^{k-1}-p^k\|^2+\frac{\gamma\tau^2}{2}\|u^{k-1}-u^{k}\|^2+\frac{\gamma\tau^2+\tau}{2}\|u^k-u^{k+1}\|^2.
\end{eqnarray}
Combining~\eqref{eq:primal_5}-\eqref{eq:primal_7}, we obtain
\begin{eqnarray*}
\Delta^k&\leq&\frac{\gamma\tau^2+L}{2}\|u^{k-1}-u^{k}\|^2-\left(\frac{1}{2\alpha\eta}-\frac{\gamma\tau^2+L+\tau}{2}\right)\|u^k-u^{k+1}\|^2\\
&&+\frac{\tau}{2}\|p^{k-1}-p^k\|^2+\frac{1}{2\alpha}\bigg{[}\frac{1}{1-\eta}\|u-v^k\|^2-\frac{1}{1-\eta}\|u-v^{k+1}\|^2-\frac{1}{\eta}\|u^k-v^k\|^2\\
&&-\left(\frac{\eta^2+\eta-1}{\eta}\right)\|v^k-u^{k+1}\|^2\bigg{]}.
\end{eqnarray*}
Since $\frac{\sqrt{5}-1}{2}\leq\eta<1$, we have $\eta^2+\eta-1\geq0$. Additionally, since $0<\alpha\leq\frac{1}{2(\gamma\tau^2+L+\tau)\eta}$, we have $\frac{1}{2\alpha\eta}-\frac{\gamma\tau^2+L+\tau}{2}\geq\frac{\gamma\tau^2+L+\tau}{2}$. Therefore, the above inequality further leads to
\begin{eqnarray}
\Delta^k&\leq&\frac{\gamma\tau^2+L+\tau}{2}\left(\|u^{k-1}-u^{k}\|^2-\|u^k-u^{k+1}\|^2\right) %\nonumber\\
 + \frac{1}{2\alpha}\bigg{[}\frac{1}{1-\eta}\|u-v^k\|^2-\frac{1}{1-\eta}\|u-v^{k+1}\|^2\bigg{]}\nonumber\\
&&+\frac{\tau}{2}\|p^{k-1}-p^k\|^2-\frac{1}{2\alpha\eta}\|u^k-v^k\|^2-\frac{\tau}{2}\|u^{k-1}-u^{k}\|^2\nonumber\\
&\overset{\mbox{\tiny \eqref{eq:primal_v3}}}{=}&\frac{\gamma\tau^2+L+\tau}{2}\left(\|u^{k-1}-u^{k}\|^2-\|u^k-u^{k+1}\|^2\right) %\nonumber\\
   +\frac{1}{2\alpha}\bigg{[}\frac{1}{1-\eta}\|u-v^k\|^2-\frac{1}{1-\eta}\|u-v^{k+1}\|^2\bigg{]} \nonumber\\
&& +\frac{\tau}{2}\|p^{k-1}-p^k\|^2-\frac{\eta}{2\alpha}\|u^k-v^{k-1}\|^2-\frac{\tau}{2}\|u^{k-1}-u^{k}\|^2.\nonumber\\
&\leq&\Lambda^k(u,p)-\Lambda^{k+1}(u,p)-\frac{1}{2\gamma}\left(\|p-p^{k-1}\|^2-\|p-p^{k}\|^2\right)\nonumber\\
&&+\frac{\tau}{2}\|p^{k-1}-p^k\|^2-\frac{\eta}{2\alpha}\|u^k-v^{k-1}\|^2-\frac{\tau}{2}\|u^{k-1}-u^{k}\|^2\nonumber\\
&\overset{\mbox{\tiny \eqref{eq:primal_v2}}}\leq&\Lambda^k(u,p)-\Lambda^{k+1}(u,p)-\frac{1}{2\gamma}\left(\|p-p^{k-1}\|^2-\|p-p^{k}\|^2\right)\nonumber\\
&&+\frac{\tau}{2}\|p^{k-1}-p^k\|^2-\frac{\eta}{2\alpha(1-\eta)^2}\|v^{k-1}-v^k\|^2-\frac{\tau}{2}\|u^{k-1}-u^{k}\|^2.
\end{eqnarray}
Recall that $\Delta^k=L_{u^*}(u^{k},q^k)-L_{u^*}(u,q^k)+\langle G(u^k)-G(u^*),u^k-u\rangle$.
We have %The above inequality gives us
\begin{eqnarray}\label{eq:primal_10}
&&L_{u^*}(u^{k},q^k)-L_{u^*}(u,q^k)+\langle G(u^k)-G(u^*),u^k-u\rangle\nonumber\\
&\leq&\Lambda^k(u,p)-\Lambda^{k+1}(u,p)-\frac{1}{2\gamma}\left(\|p-p^{k-1}\|^2-\|p-p^{k}\|^2\right)\nonumber\\
&&+\frac{\tau}{2}\|p^{k-1}-p^k\|^2-\frac{\eta}{2\alpha(1-\eta)^2}\|v^{k-1}-v^k\|^2-\frac{\tau}{2}\|u^{k-1}-u^{k}\|^2.
\end{eqnarray}

\emph{Step 2: Establish a bound on $L_{u^*}(u^k,p)-L_{u^*}(u^{k},q^k)$.}

First, recall the property of projection $\Pi_{\mathcal{S}}(x)$ on a convex set $\mathcal{S}$ (cf.~e.g.~\cite{Projectiononconvexsets}):
\begin{eqnarray}
\langle y-\Pi_{\mathcal{S}}(x), x-\Pi_{\mathcal{S}}(x)\rangle\leq0, \forall x\in\RR^m, \, \forall y\in\mathcal{S}.\label{eq:Projecproperty3_VIP}
\end{eqnarray}
%Then we can derive two inequalities.
Setting $x=p^{k-1}+\gamma\Theta(u^{k})$, $y=p$, $\mathcal{S}=\mathbf{C}^*$, we have
\begin{eqnarray*}
\langle p-p^k, p^{k-1}+\gamma\Theta(u^{k})-p^k\rangle\leq0.
\end{eqnarray*}
It follows that
\begin{eqnarray}\label{eq:dual_1}
\langle p-p^k, \Theta(u^{k})\rangle&\leq&\frac{1}{\gamma}\langle p-p^k, p^k-p^{k-1}\rangle\nonumber\\
&=&\frac{1}{2\gamma}\left[\|p-p^{k-1}\|^2-\|p-p^k\|^2-\|p^{k-1}-p^k\|^2\right].
\end{eqnarray}
Again using~\eqref{eq:Projecproperty3_VIP} with $x=p^k+\gamma\Theta(u^{k})$, $y=p^k$, we have
\begin{eqnarray*}
\langle p^k-q^{k}, p^k+\gamma\Theta(u^{k})-q^k\rangle\leq 0.
\end{eqnarray*}
It follows that
\begin{eqnarray}\label{eq:dual_2}
\langle p^k-q^{k}, \Theta(u^{k})\rangle\leq-\frac{1}{\gamma}\|p^k-q^k\|^2.
\end{eqnarray}
Then, it %the term $L_{u^*}(u^k,p)-L_{u^*}(u^{k},q^k)$
follows from~\eqref{eq:dual_1} and~\eqref{eq:dual_2} that
\begin{eqnarray}\label{eq:dual_3}
&&L_{u^*}(u^k,p)-L_{u^*}(u^{k},q^k)\nonumber\\
&=&\langle p-q^{k}, \Theta(u^{k})\rangle\nonumber\\
&=&\langle p-p^{k}, \Theta(u^{k})\rangle+\langle p^k-q^{k}, \Theta(u^{k})\rangle\nonumber\\
&\leq&\frac{1}{2\gamma}\left[\|p-p^{k-1}\|^2-\|p-p^k\|^2-\|p^{k-1}-p^k\|^2\right]-\frac{1}{\gamma}\|p^k-q^k\|^2.
\end{eqnarray}

\emph{Step 3: Finally, bound $\Lambda^k(u,p)-\Lambda^{k+1}(u,p)$.}

Summing up~\eqref{eq:primal_10} and~\eqref{eq:dual_3}, we obtain
\begin{eqnarray}\label{eq:sum_1}
&&\Lambda^k(u,p)-\Lambda^{k+1}(u,p)\nonumber\\
&\geq &L_{u^*}(u^{k},p)-L_{u^*}(u,q^k)+\langle G(u^k)-G(u^*),u^k-u\rangle +\frac{1-\tau\gamma}{2\gamma}\|p^{k-1}-p^k\|^2\nonumber\\
&&+\frac{1}{\gamma}\|p^k-q^k\|^2+\frac{\eta}{2\alpha(1-\eta)^2}\|v^{k-1}-v^k\|^2+\frac{\tau}{2}\|u^{k-1}-u^{k}\|^2\nonumber\\
&\geq &L_{u^*}(u^{k},p)-L_{u^*}(u,q^k)+\langle G(u^k)-G(u^*),u^k-u\rangle\nonumber\\
&&+\rho \left[\|w^{k-1}-w^k\|^2+\|p^k-q^k\|^2\right],
\end{eqnarray}
which is the desired inequality.
\end{proof}

\subsection{Global convergence}\label{subsec:convergence}
Based on the %Lyapunov function
 analysis of Subsection~\ref{subsec:Lyapnov}, we are now in a position to establish
 the global convergence property of ALAVI.
\begin{assumption} \label{AssumptionB}
Suppose that (VIS) satisfies the primal-dual variational coherence condition.
\end{assumption}
\begin{theorem}[Global convergence of ALAVI method]\label{theo:convergence_VIP}
In addition of the assumptions in Lemma~\ref{lemma:bound1_VIP}, we suppose that Assumption~\ref{AssumptionB} holds. Let $u^*$ be a solution of (VIP), $(u^*,p^*)$ be a saddle point of $L_{u^*}(u,p)$ over $\mathbf{U}\times\mathbf{C}^*$, the sequence $\{(v^k,u^{k},p^{k})\}$ is generated by ALAVI. Then, the following assertions hold true:
\begin{itemize}
\item[{\rm(i)}] The sequence $\{(v^k,u^{k},p^{k})\mid k=1,2,...\}$ is bounded.
\item[{\rm(ii)}] Every cluster point of $\{(v^k,u^{k},p^{k})\mid k=1,2,...\}$ is a solution of (VIS).
\item[{\rm(iii)}] The entire sequence $\{(v^k,u^{k},p^{k})\mid k=1,2,...\}$ converges to a solution of (VIS).
\end{itemize}
\end{theorem}
\begin{proof}$\quad$
\begin{itemize}
\item[{\rm(i)}] Take $(u,p)=(u^*,p^*)\in\setS_{\mbox{\tiny VIS}}$, and denote $W_k:=\rho \left[\|w^{k-1}-w^k\|^2+\|p^k-q^k\|^2\right]$. We have
\begin{eqnarray}\label{eq:convergence}
&&\Lambda^k(u^*,p^*)-\Lambda^{k+1}(u^*,p^*)\nonumber\\
&\overset{\mbox{\tiny Lem.~\ref{lemma:bound1_VIP}}}{\geq}&\left[L_{u^*}(u^{k},p^*)-L_{u^*}(u^*,q^k)\right]+\langle G(u^k)-G(u^*),u^k-u^*\rangle +W_k \nonumber \\
&=&\langle G(u^k),u^k-u^*\rangle+J(u^k)-J(u^*)+\langle p^*,\Theta(u^k)\rangle-\langle q^k,\Theta(u^*)\rangle + W_k\nonumber \\
                                         &\geq&\langle G(u^k),u^k-u^*\rangle+J(u^k)-J(u^*)+\langle p^*,\Theta(u^k)-\Theta(u^*)\rangle +W_k \nonumber \\
                                         &\overset{\mbox{\tiny Asm.~\ref{AssumptionB} \& Prop.~\ref{Minty})}}{\geq}&  W_k \nonumber \\ %\qquad\mbox{(by Assumption B and~\ref{Minty})}\nonumber\\
                                         &=& \rho \left[\|w^{k-1}-w^k\|^2+\|p^k-q^k\|^2\right]. %\nonumber
\end{eqnarray}
 If $v^{k-1}=v^k$ (or $v^{k}=v^{k+1}$), $u^{k-1}=u^k$ (or $u^{k}=u^{k+1}$), $p^{k-1}=p^k$ and $p^k=q^k$, then by~\eqref{eq:primal_v2},~\eqref{eq:VI} and~\eqref{eq:dual_1}, it follows straightforwardly that $(u^k,p^k)$ is a solution of (VIS). Otherwise, the nonnegative sequence $\{\Lambda^k(u^*,p^*)\mid k=1,2,...\}$ is strictly decreasing and converges to a limit. Hence,  $\{\Lambda^k(u^*,p^*)\mid k=1,2,...\}$ is bounded,
  which implies that the sequence $\{(v^k,p^k)\mid k=1,2,...\}$ is bounded by the definition of $\Lambda^k(u^*,p^*)$ (cf.~\eqref{eq:Lambda_VIP}). Now, combining boundedness of $\{(v^k,p^k)\mid k=1,2,...\}$ and Step~\eqref{ALAVI_relax} of ALAVI, it follows that the sequence $\{u^k\mid k=1,2,...\}$ is bounded as well. Thus, we obtain boundedness of the sequence $\{(v^k,u^k,p^k)\mid k=1,2,...\}$ as well.
 %{\color{red} (We may need to argue why $\{w^k\mid k=1,2,...\}$ is bounded. In principle, even if $\sum_k \|w^{k-1}-w^k\|^2$ is convergent it does not imply $\{w^k\mid k=1,2,...\}$ is bounded; e.g.\ $w^k=1+1/2+\cdots+1/k$.} {\color{brown} We modified the proof and marked them in brown. Please check.}{\color{red})}

\item[{\rm(ii)}] Based on (i), the sequence $\{(v^k,u^k,p^k)\mid k=1,2,...\}$ has a cluster point $(\bar{v},\bar{u},\bar{p})$, and let us assume the subsequence $\{(v^{k'},u^{k'},p^{k'})\}\rightarrow(\bar{v},\bar{u},\bar{p})$. Passing the limit on this subsequence in~\eqref{ALAVI_relax},~\eqref{eq:VI} and~\eqref{eq:dual_1} respectively, it follows that $\bar{v}=\bar{u}$, $\bar{u}\in\UU$ and $\bar{p}\in\mathbf{C}^*$,
\[
\langle G(\bar{u}),u-\bar{u}\rangle+J(u)-J(\bar{u})+\langle\bar{p},\Theta(u)-\Theta(\bar{u})\rangle\geq 0,\quad\forall u\in\UU,
\]
and
\[
\langle p-\bar{p}, \Theta(\bar{u})\rangle\leq0,\quad\forall p\in\mathbf{C}^*.
\]
Thus, we can conclude that $(\bar{u},\bar{p})$ (or $(\bar{v},\bar{p})$) is a KKT point of (VIP) over $\UU\times\mathbf{C}^*$.
\item[{\rm(iii)}] The previous convergence analysis holds for any solution $(u^*,p^*)$ of (VIS). Thus, taking $(u^*,p^*)=(\bar{u},\bar{p})$ we have
%we can show that $\Lambda^k(\bar{u},\bar{p})\rightarrow\nu$ for some $\nu\ge 0$, where
\[
\Lambda^k(\bar{u},\bar{p})=\frac{1}{2\alpha(1-\eta)}\|v^k-\bar{u}\|^2+\frac{\gamma\tau^2+L+\tau}{2}\|u^{k-1}-u^k\|^2+\frac{1}{2\alpha}\|p^{k-1}-\bar{p}\|^2.
\]
Moreover, $\Lambda^k(\bar{u},\bar{p})-\Lambda^{k+1}(\bar{u},\bar{p})\rightarrow0$. By \eqref{eq:convergence}, we have that
 $\|u^k-v^{k+1}\|\rightarrow0$, $\|u^{k+1}-u^k\|\rightarrow0$, and $\|p^k-q^k\|\rightarrow0$.
Therefore, $\Lambda^k(\bar{u},\bar{p})\rightarrow 0$, implying that
the entire sequence $\{(v^k,u^{k},p^{k})\mid k=1,2,...\}$ converges to $(\bar{u},\bar{u},\bar{p})$.
\end{itemize}
\end{proof}
\begin{remark} If $\Theta$ is non-affine (but differentiable), then the primal subproblem~\eqref{ALAVI_primal} can be replaced by
\[
u^{k+1}=\arg\min_{u\in\UU} \,\, \langle G(u^k),u\rangle+J(u)+\langle q^k,\nabla\Theta(u^k)\cdot u\rangle+\frac{1}{2\alpha}\|u-v^k\|^2.
\]
%assuming that $\Theta$ is differentiable.
We shall not give a formal analysis for that variation of the method. It suffices to say that after adapting steps~\eqref{ALAVI_relax}-\eqref{ALAVI_dual} in ALAVI accordingly, a convergence result like Theorem~\ref{theo:convergence_VIP} can be established similarly.
%Using this algorithm, we can get a modified augmented Lagrangian decomposition method for the nonseparable case.
\end{remark}

%\section{The KKT Stationary Measure for (VIP) and Iteration Complexity of ALAVI}
\section{Iteration Complexity Analysis for ALAVI}
\label{sec:rate}
\subsection{KKT stationary measure for (VIP)}\label{subsec:KKT}
In Section~\ref{sec:pre}, we show that a solution of (VIS) or a KKT point for (VIP) satisfies the following KKT system
\begin{equation}\label{eq:KKT2}
\left\{
\begin{array}{l}
0\in G(u)+\partial J(u) +\left(\partial\Theta(u)\right)^{\top}p+\mathcal{N}_{\mathbf{U}}(u)      \\
0\in-\Theta(u)+\mathcal{N}_{\mathbf{C}^*}(p).
\end{array}
\right.
\end{equation}
We define the Lagrangian function-based mapping $H(u,p)=\left(
\begin{array}{l}
G(u)+\partial J(u) +\left(\partial\Theta(u)\right)^{\top}p+\mathcal{N}_{\mathbf{U}}(u)      \\
-\Theta(u)+\mathcal{N}_{\mathbf{C}^*}(p)
\end{array}
\right)$ to be the KKT mapping for (VIP). By ALAVI, at iteration $k-1$, we have
\begin{equation}
\left\{
\begin{array}{l}
0\in G(u^{k-1})+\partial J(u^{k}) +\big{(}\partial\Theta(u^{k})\big{)}^{\top} q^{k-1}+\frac{1}{\alpha}\left[u^{k}-v^{k-1}\right]+\mathcal{N}_{\mathbf{U}}(u^{k})      \\
0\in-\Theta(u^{k})+\frac{1}{\gamma}\left[p^{k}-p^{k-1}\right]+\mathcal{N}_{\mathbf{C}^*}(p^{k}) .
\end{array}
\right.
\end{equation}
Thus,
\begin{eqnarray*}
\xi^{k}=\left(
\begin{array}{l}
G(u^{k})-G(u^{k-1}) +\big{(}\theta^{k}\big{)}^{\top}(p^{k}-q^{k-1})+\frac{1}{\alpha}\left[v^{k-1}-u^{k}\right]      \\
\frac{1}{\gamma}\left[p^{k-1}-p^{k}\right]
\end{array}
\right)\in H(u^{k},p^k)
\end{eqnarray*}
with $\theta^{k}\in\partial\Theta(u^{k})$. By Assumption~\ref{assumpA}, if we denote $c_1:=3(L+\gamma\tau^2)^2$, $c_2:=\frac{3}{\alpha^2}$ and $c_3:=\frac{3}{\gamma^2}$, then
\begin{eqnarray*}
\|\xi^{k}\|^2&\leq&c_1\|u^{k-1}-u^{k}\|^2+c_2\|u^k-v^{k-1}\|^2+c_3\|p^{k-1}-p^{k}\|^2\nonumber \\
&\overset{\mbox{\tiny \eqref{eq:primal_v2}}}\leq& c_1\|u^{k-1}-u^{k}\|^2+\frac{c_2}{(1-\eta)^2}\|v^{k-1}-v^k\|^2+c_3\|p^{k-1}-p^{k}\|^2 \nonumber\\
                               &\leq&\sigma^2\|w^{k-1}-w^k\|^2,
\end{eqnarray*}
where $\sigma^2:=\max\{c_1, \frac{c_2}{(1-\eta)^2},c_3\}$. The KKT stationary measure is given by following
\begin{equation}\label{eq:stationary_measure}
\dist(0,H(u^k,p^k))\leq\sigma\|w^{k-1}-w^k\|.
\end{equation}

\subsection{Non-ergodic rate of convergence under the primal-dual variational coherence condition for (VIS)}\label{subsec:nonergodic_rate}

In the proof of Theorem~\ref{theo:convergence_VIP}, we derived
\begin{equation}\label{eq:descent_rate}
\Lambda^k(u^*,p^*)-\Lambda^{k+1}(u^*,p^*)\geq \rho \|w^{k-1}-w^k\|^2,
\end{equation}
and $\Lambda^k(u^*,p^*) \rightarrow 0 $, as $k\rightarrow \infty$. %Hence, $\rho \sum_{k=1}^t\|w^{k-1}-w^k\|^2\leq\Lambda^1(u^*,p^*).$

\begin{theorem}[$o(1/\sqrt{k})$ nonergodic convergence rate] Under the setting of Theorem~\ref{theo:convergence_VIP}, the sequence $\{(v^k,u^k,p^k)\}$ is generated by ALAVI. Then, we have
\[
\lim_{k\rightarrow \infty}\inf\sqrt{k}\|w^{k-1}-w^k\|=0
\]
i.e., the asymptotic $o(1/\sqrt{k})$ convergence rate holds for $\|w^{k-1}-w^k\|$ and $\dist(0,H(u^k,p^k))$ has $o(1/\sqrt{k})$ convergence rate.
\end{theorem}

%{\color{red} (I changed the statement and the proof. Please check.)}
\begin{proof}
For any given integer $K$, we have
\begin{eqnarray*}
& & \rho K \min_{K+1 \le k \le 2K}  \|w^{k-1}-w^k\|^2 \\
&\le& \rho \sum_{k=K+1}^{2K} \|w^{k-1}-w^k\|^2 \\
&\le& \Lambda^{K+1} (u^*,p^*) \rightarrow 0, \mbox{ as } K \rightarrow \infty,
\end{eqnarray*}
which implies
\[
\lim_{k\rightarrow \infty}\inf\sqrt{k}\|w^{k-1}-w^k\|=0,
\]
as desired.
%Suppose by contradiction that this statement does not hold. That is,
%\[
%\lim_{k\rightarrow \infty}\inf\sqrt{k+1}\|w^{k-1}-w^k\| \geq \delta, %\qquad\mbox{for some $\delta>0$}
%\]
%for some $\delta>0$. Then, for $k_0$ large enough and for all $k\geq k_0$ we have
%\[
%\|w^{k-1}-w^k\|\geq\frac{\delta}{\sqrt{k+1}}>0.
%\]
%Therefore,
%\[
%\sum_{k=0}^{\infty}\|w^{k-1}-w^k\|^2\geq\sum_{k=k_0}^{\infty}\|w^{k-1}-w^k\|^2\geq\sum_{k=k_0}^{\infty}\frac{\delta^2}{k+1}
%\]
%Since the fact that $\sum_{k=k_0}^{\infty}\frac{\delta^2}{k+1}=\infty$, we have that above inequality is contradicted with the fact that $\rho \sum_{k=1}^t\|w^{k-1}-w^k\|^2\leq\Lambda^1(u^*,p^*)$. Therefore, we obtain the desired results. Finally, the fact that $\dist(0,H(u^k,p^k))\leq\sigma\|w^{k-1}-w^k\|$ implies an $O(1/\sqrt{k})$ convergence rate for $\dist^2(0,H(u^k,p^k))$.
\end{proof}

%\begin{theorem}[Iteration complexity] Under the setting of Lemma~\ref{lemma:bound1_VIP}, suppose that the sequence $\{(v^k,u^k,p^k)\}$ is generated by ALAVI. Then,
%\[
%\min_{1\leq k\leq t}\|w^{k-1}-w^k\|\leq\sqrt{\frac{\Lambda^1(u^*,p^*)}{\rho t}}.
%\]
%Consequently, to achieve an $\varepsilon$-KKT point of (VIP), ALAVI needs at most $t=\frac{\Lambda^1(u^*,p^*)}{\rho \varepsilon^2}$ iterations.
%\end{theorem}
%
%\begin{proof} Recalling the fact that $\rho \sum\limits_{k=1}^{\infty}\|w^{k-1}-w^k\|^2\leq\Lambda^1(u^*,p^*)$. Since $\|w^{k-1}-w^k\|\geq0$ and $h(x)=x^2$, $x\in\RR$ is monotone increasing in $x$ on $[0,\infty)$, we have that
%\[
%t \rho \left(\min_{1\leq k\leq t}\|w^{k-1}-w^k\|\right)^2  \leq\Lambda^1(u^*,p^*),
%\]
%leading to the desired result.
%\end{proof}

\subsection{Ergodic rate when $G$ is monotone }\label{subsec:ergodic_rate}

In order to derive an $O(1/k)$ ergodic iteration rate for ALAVI, we now assume $G$ to be monotone on $\UU$. By Theorem~\ref{theo:convergence_VIP}, we conclude that the sequence $\{(v^k,u^k,p^k)\mid k=1,2,...\}$ is bounded. Therefore, there exists a positive number $\tilde{r}$ such that for all $k\in\mathbb{N}$, $\|w^k\|\leq\tilde{r}$. Moreover, we have that
\[
\|q^k\|\leq\|q^k-p^{k+1}\|+\|p^{k+1}\|\leq\gamma\tau\|u^k-u^{k+1}\|+\|p^{k+1}\|\leq(2\gamma\tau+1)\tilde{r}.
\]
Let $r:=(2\gamma\tau+1)\tilde{r}$. Then we have that there exists two balls $\mathfrak{B}_u(r)$ and $\mathfrak{B}_p(r)$ centered at the origin (in two different dimensional spaces though) and the same radius $r$ such that $u^k\in\mathfrak{B}_u(r)$ and $q^k\in\mathfrak{B}_p(r)$, for any $k\in\mathbb{N}$. Additionally, by the convergence analysis of Theorem~\ref{theo:convergence_VIP}, we have that $(\bar{u},\bar{p})\in\mathfrak{B}_u(r)\times\mathfrak{B}_p(r)$. We define the average of the sequence $\{(v^k,u^k,p^k)\mid k=1,2,...\}$ generated from Algorithm ALAVI, converges to the saddle point $(u^*,p^*)$; that is, for and any integer $t>0$, define
\[
\tilde{u}_{t}:=\frac{\sum_{k=1}^{t}u^{k}}{t}\;\mbox{and}\;\tilde{p}_{t}:=\frac{\sum_{k=1}^{t}q^k}{t}.
\]
Obviously, $(\tilde{u}_{t},\tilde{p}_{t})\in\mathfrak{B}_u(r)\times\mathfrak{B}_p(r)$. Define %the bifunction
\[
\Psi(u,p,\zeta,\lambda):=\langle G(\zeta), u-\zeta\rangle+J(u)-J(\zeta)+\langle\lambda,\Theta(u)\rangle-\langle p,\Theta(\zeta)\rangle.
\]
A primal-dual gap function for (VIP) can be defined as
\[
\psi(u,p) := \max_{(\zeta,\lambda)\in(\UU\cap\mathfrak{B}_u(r))\times(\mathbf{C}^*\cap\mathfrak{B}_p(r))}\Psi(u,p,\zeta,\lambda).
\]
Next lemma gives some basic property of the gap function $\psi$ under Assumption~\ref{AssumptionB}.

\begin{lemma} Under the same setting as in Theorem~\ref{theo:convergence_VIP}, we have the following properties:
\begin{itemize}
\item[{\rm(i)}] %{\color{blue} $\psi(u,p)$ and $\Psi(u,p,\zeta,\lambda)$ are both convex in $u$ and linear in $p$.}
$\Psi(u,p,\zeta,\lambda)$ is convex in $u$ and linear in $p$ whenever $\lambda\in\mathbf{C}^*$, and $\psi(u,p)$ is convex in $(u,p)$.
\item[{\rm(ii)}] If $(u,p)\in(\UU\cap\mathfrak{B}_u(r))\times(\mathbf{C}^*\cap\mathfrak{B}_p(r))$, then $\psi(u,p)\geq0$.
\item[{\rm(iii)}] If $(u^*,p^*)\in(\UU\cap\mathfrak{B}_u(r))\times(\mathbf{C}^*\cap\mathfrak{B}_p(r))$ is a KKT point of (VIP), then $\psi(u^*,p^*)=0$.
\item[{\rm(iv)}] For $(\hat{u},\hat{p})\in(\UU\cap\mathfrak{B}_u(r))\times(\mathbf{C}^*\cap\mathfrak{B}_p(r))$ such that $\psi(\hat{u},\hat{p})=0$, $(\hat{u},\hat{p})$ is a KKT point of (VIP).
\end{itemize}
\end{lemma}

\begin{proof} {\ }

\begin{itemize}
\item[{\rm(i)}] Trivial. %{\color{red} (Why is $\psi(u,p)$ linear in $p$? It looks like only convex in $p$.)}
\item[{\rm(ii)}] By the definition of $\psi(u,p)$ and~\eqref{VIS2_VIP}, we have
\begin{eqnarray*}
\psi(u,p)\geq\langle G(u^*), u-u^*\rangle+J(u)-J(u^*)+\langle p^*,\Theta(u)\rangle-\langle p,\Theta(u^*)\rangle\geq0.
\end{eqnarray*}
\item[{\rm(iii)}] By the definition of $\psi(u,p)$ and Assumption~\ref{AssumptionB}, we have
\begin{eqnarray*}
\psi(u^*,p^*)&=&\max_{\scriptsize (\zeta,\lambda)\in(\UU\cap\mathfrak{B}_u(r))\times(\mathbf{C}^*\cap\mathfrak{B}_p(r))} \left[ \langle G(\zeta), u^*-\zeta\rangle+J(u^*)-J(\zeta) \right] \\
&& \quad \quad \quad \quad \quad \quad \quad \quad \quad \quad \quad +\langle\lambda,\Theta(u^*)\rangle-\langle p^*,\Theta(\zeta)\rangle\\
&\overset{\mbox{\tiny \rm Asm.~\ref{AssumptionB} \&~\eqref{VIS2_VIP}}}{\leq}&0 . %\qquad\mbox{(by Assumption B and~\eqref{VIS2_VIP})}.
\end{eqnarray*}
Together with the fact $\psi(u,p)\geq0$ in statement (ii), we obtain $\psi(u^*,p^*)=0$.
\item[{\rm(iv)}] Since $\psi(\hat{u},\hat{p})=0$, we have that for any $(\zeta,\lambda)\in(\UU\cap\mathfrak{B}_u(r))\times(\mathbf{C}^*\cap\mathfrak{B}_p(r))$,
\begin{equation}\label{eq:bifunction}
\langle G(\zeta), \zeta-\hat{u}\rangle+J(\zeta)-J(\hat{u})+\langle\hat{p},\Theta(\zeta)\rangle-\langle\lambda,\Theta(\hat{u})\rangle\geq0.
\end{equation}
Taking $\zeta=\hat{u}$, we have
\begin{equation}\label{eq:gap_dual}
\langle\hat{p}-\lambda,\Theta(\hat{u})\rangle\geq0,\quad\forall\lambda\in\mathbf{C}^*\cap\mathfrak{B}_p(r).
\end{equation}
Taking $\lambda=\hat{p}$ in~\eqref{eq:bifunction}, we have
\[
\langle G(\zeta),\zeta-\hat{u}\rangle+J(\zeta)-J(\hat{u})+\langle\hat{p},\Theta(\zeta)-\Theta(\hat{u})\rangle\geq0\quad\forall\zeta\in\UU\cap\mathfrak{B}_u(r).
\]
Since $G$ is continuous, the above inequality yields
\begin{equation}\label{eq:gap_primal}
\langle G(\hat{u}),\zeta-\hat{u}\rangle+J(\zeta)-J(\hat{u})+\langle\hat{p},\Theta(\zeta)-\Theta(\hat{u})\rangle\geq0\quad\forall\zeta\in\UU\cap\mathfrak{B}_u(r).
\end{equation}
Combining~\eqref{eq:gap_dual} with~\eqref{eq:gap_primal}, we conclude that $(\hat{u},\hat{p})$ is a KKT point of (VIP).
\end{itemize}
\end{proof}

\begin{theorem}[$O(1/k)$ ergodic convergence rate] Under the setting of Lemma~\ref{lemma:bound1_VIP}, assume that $G$ is monotone on $\UU$. The sequence $\{(v^k,u^k,p^k)\}$ is generated by ALAVI method. Then there exists $\tilde \rho>0$ such that
\[
\psi(\tilde{u}_t,\tilde{p}_t)\leq\frac{\tilde \rho}{t}.
\]
\end{theorem}
\begin{proof}
By Lemma~\ref{lemma:bound1_VIP} and the monotonicity of $G$, we have that
\begin{eqnarray*}
&&\Lambda^k(u,p)-\Lambda^{k+1}(u,p)\\
&\geq&\left[L_{u^*}(u^k,p)-L_{u^*}(u,q^k)\right]+\langle G(u^k)-G(u^*),u^k-u\rangle\\
&=&\langle G(u^k),u^k-u\rangle+J(u^k)-J(u)+\langle p,\Theta(u^k)\rangle-\langle q^k,\Theta(u)\rangle\\
&=&\langle G(u),u^k-u\rangle+\langle G(u^k)-G(u),u^k-u\rangle+J(u^k)-J(u)\\
&&+\langle p,\Theta(u^k)\rangle-\langle q^k,\Theta(u)\rangle\\
&\overset{\mbox{\tiny \rm $G$ monotone}}{\geq}&\Psi(u^k,q^k,u,p). %\qquad\qquad\qquad\mbox{(by the monotonicity of $G$)}
\end{eqnarray*}
Since $\Psi$ is convex in $u$ and linear in $p$, we have that
\begin{eqnarray*}
\Psi(\tilde{u}_{t},\tilde{p}_{t},u,p)&\leq&\frac{1}{t}\sum_{k=1}^t\left[\langle G(u),u^k-u\rangle+J(u^k)-J(u)+\langle p,\Theta(u^k)\rangle-\langle q^k,\Theta(u)\rangle\right]\\
&\leq&\frac{\Lambda^1(u,p)}{t}.
\end{eqnarray*}
Taking $(u,p)=(u_t^+,p_t^+)=\arg\max_{(u,p)\in(\UU\cap\mathfrak{B}_u(r))\times(\mathbf{C}^*\cap\mathfrak{B}_p(r))}\Psi(\tilde{u}_t,\tilde{p}_t,u,p)$, the above inequality yields
\[
\psi(\tilde{u}_t,\tilde{p}_t)\leq\frac{1}{t}\Lambda^1(u_t^{+},p_t^+).
\]
Since $(u_t^{+},p_t^{+})\in\mathfrak{B}_u(r)\times\mathfrak{B}_p(r)$, there is $\tilde{\rho}>0$ so that $\psi(\tilde{u}_t,\tilde{p}_t)\leq\frac{1}{t}\Lambda^1(u_t^{+},p_t^{+})\leq\frac{\tilde{\rho}}{t}$.
\end{proof}

\section{Linear Convergence of ALAVI under the Metric Subregularity Condition of the KKT Mapping}\label{sec:linear}

In order to get linear convergence of ALAVI, we introduce the metric subregularity of KKT mapping $H(u,p)=H(x)$ with $x=\left(\begin{array}{c}u\\ p\end{array}\right)$.
\begin{definition}[Metric subregularity] The set-valued mapping $F(x)$ is said to be metric subregular around $(x^*,0)$ if there exist $\delta>0$ %\mathbb{B}(x^*;\delta)$ of $x^*$
and $\kappa>0$ in such a way that
\begin{equation}\label{MS}
\dist(x,F^{-1}(0))\leq\kappa \dist\left(0,F(x)\right),\quad\forall x\in \mathfrak{B}(x^*;\delta)
\end{equation}
where $\mathfrak{B}(x^*;\delta)$ is a ball center at $x^*$ with radius $\delta$.
\end{definition}

Now, let us consider the KKT mapping $H(x)$.
Suppose that $H(x)$ is metric subregular around $(\bar{x},0)$, where $\bar x$ is a KKT point of (VIP). This means that the following error bound condition for (VIP) holds:
\begin{equation}
\dist(x^k,\setS_{\mbox{\tiny VIS}})\leq\kappa \dist(0,H(u^k,p^k)),\quad\forall x^k\in \mathfrak{B}(\bar{x};\eta) .
\end{equation}
By the KKT stationary measure and the global convergence of $\{(v^k,u^k,p^k\mid k=1,2,...)\}$ generated by ALAVI as presented in Subsection~\ref{subsec:convergence}, there is $k_0$ such that for all $k\geq k_0$, the following bound hold true as well: %we have the another useful error bound of (VIP):
\[
\dist(x^k,\setS_{\mbox{\tiny VIS}}) \overset{\eqref{eq:stationary_measure}}{\leq} \kappa\sigma\|w^{k-1}-w^k\|^2 . %\qquad\mbox{(See~\eqref{eq:stationary_measure})}
\]

\begin{definition}[Weighted distance]\label{def:weight_dist}
Given point $z^k=\left(\begin{array}{c}v^k\\p^{k-1}\end{array}\right)$, the weighted distance of $z^k$ to solution set $\setS_{\mbox{\tiny VIS}}$ is defined as
\[
\widetilde{\dist}^2(z^k,\setS_{\mbox{\tiny VIS}})=\min_{(u,p)\in\setS_{\mbox{\tiny VIS}}}\left(\beta_1\|v^k-u\|^2+\beta_3\|p^{k-1}-p\|^2\right),
\]
where $\beta_1$ and $\beta_3$ are defined according to~\eqref{eq:Lambda_VIP}.
\end{definition}
\begin{lemma}\label{lemma:bound2_VIP} If KKT mapping $H(x)$ satisfies metric subregularity around $(\bar{x},0)$ with $\kappa$ and $\delta$, there is $k_0$, for $k\geq k_0$, the following inequality holds
\[
\widetilde{\dist}^2(z^k,\setS_{\mbox{\tiny VIS}})+\beta_2\|u^{k-1}-u^k\|^2\leq\bar{\beta}\|w^{k-1}-w^k\|^2,\qquad\mbox{for $k\geq k_0$,}
\]
where $\beta_2$ is as given in~\eqref{eq:Lambda_VIP} and $\bar{\beta} := 2\max\{\beta_1,\beta_3\}\kappa^2\sigma^2+\max\{2\beta_1,\beta_2,2\beta_3\}$.
\end{lemma}
\begin{proof}
By~\eqref{def:weight_dist} and the metric subregularity of $H(x)$, we have
\begin{eqnarray}
&&\widetilde{\dist}^2(z^k,\setS_{\mbox{\tiny VIS}})+\beta_2\|u^{k-1}-u^k\|^2\nonumber\\
&=&\min_{(u,p)\in\setS_{\mbox{\tiny VIS}}}\left\{\beta_1\|v^k-u\|^2+\beta_3\|p^{k-1}-p\|^2\right\}+\beta_2\|u^{k-1}-u^k\|^2\nonumber\\
&\leq&\min_{(u,p)\in\setS_{\mbox{\tiny VIS}}}\left\{2\beta_1\|u^k-u\|^2+2\beta_1\|u^{k}-v^k\|^2+2\beta_3\|p^{k}-p\|^2+2\beta_3\|p^{k-1}-p^k\|^2\right\}\nonumber\\
&&+\beta_2\|u^{k-1}-u^k\|^2\nonumber\\
&\leq&2\max\{\beta_1,\beta_3\}\dist^2(x^k,\setS_{\mbox{\tiny VIS}})+2\beta_1\|u^k-v^k\|^2+2\beta_3\|p^{k-1}-p^k\|^2\nonumber\\
&&+\beta_2\|u^{k-1}-u^k\|^2\nonumber\\
&\leq&2\max\{\beta_1,\beta_3\}\kappa^2\dist^2(0,H(u^k,p^k))+2\beta_1\|u^k-v^k\|^2+2\beta_3\|p^{k-1}-p^k\|^2\nonumber\\
&&+\beta_2\|u^{k-1}-u^k\|^2\nonumber\\
&&\mbox{(by the metric subregularity of $H$ and the convergence results of~\ref{subsec:convergence})}\nonumber\\
&\overset{\mbox{\tiny \eqref{eq:stationary_measure}}}{\leq}&2\max\{\beta_1,\beta_3\}\kappa^2\sigma^2\|w^{k-1}-w^k\|^2+2\beta_1\|u^k-v^k\|^2+2\beta_3\|p^{k-1}-p^k\|^2\nonumber\\
&&+\beta_2\|u^{k-1}-u^k\|^2 \nonumber \\ %\qquad\qquad\qquad\qquad\qquad\;\mbox{(by the KKT stationary measure~\eqref{eq:stationary_measure})}\nonumber\\
&\overset{\mbox{\tiny \eqref{eq:primal_v3})}}{\leq} &2\max\{\beta_1,\beta_3\}\kappa^2\sigma^2\|w^{k-1}-w^k\|^2+\frac{2\beta_1\eta^2}{(1-\eta)^2}\|v^{k-1}-v^k\|^2+2\beta_3\|p^{k-1}-p^k\|^2\nonumber\\
&&+\beta_2\|u^{k-1}-u^k\|^2 \nonumber \\
%\qquad\qquad\qquad\qquad\qquad\qquad\qquad\qquad\qquad\qquad\qquad\quad\;\mbox{(by~\eqref{eq:primal_v3})}\nonumber\\
&\leq&\bar{\beta}\|w^{k-1}-w^k\|^2.
\end{eqnarray}
%with $\bar{\beta}=2\max\{\beta_1,\beta_3\}\kappa^2\sigma^2+\max\{\frac{2\beta_1\eta^2}{(1-\eta)^2},\beta_2,2\beta_3\}$. The desired result follows.
\end{proof}
By the estimation of the change of $\Lambda^k(u,p)$ (\eqref{eq:Lambda_VIP} and Lemma~\ref{lemma:bound1_VIP}) with $x={\rm Proj}_{\setS_{\mbox{\tiny VIS}}}(z^k)=\left(\begin{array}{c}v_k^*\\p_{k-1}^*\end{array}\right)$, it follows that
\begin{eqnarray}\label{eq:variance_Lambda}
&&\Lambda^{k+1}(v_k^*,p_{k-1}^*)+\left[L_{v_k^*}(u^k,p_{k-1}^*)-L_{v_k^*}(v_k^*,q^{k})\right]\nonumber\\
&\overset{\mbox{\tiny Lem.~\ref{lemma:bound1_VIP}}}{\leq} &\Lambda^{k}(v_k^*,p_{k-1}^*)-\rho \|w^{k-1}-w^k\|^2\nonumber\\
&\overset{\mbox{\tiny Def.~\ref{def:weight_dist}}}{=}&\widetilde{\dist}^2(z^{k},\setS_{\mbox{\tiny VIS}})+\beta_2\|u^{k-1}-u^{k}\|^2- \rho \|w^{k-1}-w^k\|^2.
%\;\mbox{(by~\ref{def:weight_dist})}\nonumber\\
\end{eqnarray}
On the other hand, by the definition of the Lyapnov function value $\Lambda^k(u,p)$ and the fact that $(v_k^*,p_{k-1}^*)$ is a saddle point of $L_{v_k^*}(u,p)$, we have
\begin{eqnarray}\label{eq:linear_converge}
&&\widetilde{\dist}^2(z^{k+1},\setS_{\mbox{\tiny VIS}})+\beta_2\|u^{k}-u^{k+1}\|^2\nonumber\\
&\leq&\Lambda^{k+1}(v_k^*,p_{k-1}^*)+\left[L_{v_k^*}(u^k,p_{k-1}^*)-L_{v_k^*}(v_k^*,q^{k})\right]\nonumber\\
&\overset{\mbox{\tiny \eqref{eq:variance_Lambda}}}{\leq}&\widetilde{\dist}^2(z^{k},\setS_{\mbox{\tiny VIS}})+\beta_2\|u^{k-1}-u^{k}\|^2- \rho \|w^{k-1}-w^k\|^2\nonumber\\
&\overset{{\mbox{\tiny Lem.~\ref{lemma:bound2_VIP}}}}{\leq}&\widetilde{\dist}^2(z^{k},\setS_{\mbox{\tiny VIS}})+\beta_2\|u^{k-1}-u^{k}\|^2-\frac{\rho}{\bar{\beta}}\left[\widetilde{\dist}^2(z^{k},\setS_{\mbox{\tiny VIS}})+\beta_2\|u^{k-1}-u^k\|^2\right]\nonumber\\
%&&\qquad\qquad\qquad\qquad\qquad\qquad\qquad\qquad\qquad\qquad\qquad\qquad\quad\mbox{(by~\ref{lemma:bound2_VIP})}\nonumber\\
&=&\left(1-\frac{\rho}{\bar{\beta}}\right)\left[\widetilde{\dist}^2(z^{k},\setS_{\mbox{\tiny VIS}})+\beta_2\|u^{k-1}-u^{k}\|^2\right].
\end{eqnarray}
A linear rate of convergence for the sequence $\widetilde{\dist}^2(z^{k},\setS_{\mbox{\tiny VIS}})+\beta_2\|u^{k-1}-u^{k}\|^2$ thus follows from the above inequality.
%unless $z^{k+1}\in\setS_{\mbox{\tiny VIS}}$ and $u^k=u^{k+1}$.
The above convergence result is summerized in the following theorem.

\begin{theorem}[Linear convergence of ALAVI]\label{theo:linear_converg} Suppose that the same conditions as in Theorem~\ref{theo:convergence_VIP} hold. Let $\{(v^k,u^k,p^k)\mid k=1,2,...\}$ be the sequence generated by ALAVI converging to $(\bar{u},\bar{u},\bar{p})$ and the KKT mapping $H(u,p)$ is metric subregular around $((\bar{u},\bar{p}),0)$ with $\kappa>0$ and $\delta>0$. Then,  the sequence $\{(v^k,u^k,p^k)\mid k=1,2,...,\}$ %generated by algorithm ALAVI
actually converges to $(\bar{u},\bar{u},\bar{p})$ at a linear rate.
\end{theorem}

%Furthermore, by~\ref{theo:linear_converg} and
Continuing from~\eqref{eq:linear_converge},
we have
\begin{eqnarray*}
\widetilde{\dist}^2(z^{k+1},\setS_{\mbox{\tiny VIS}})&\leq&\widetilde{\dist}^2(z^{k+1},\setS_{\mbox{\tiny VIS}})+\beta_2\|u^{k}-u^{k+1}\|^2\\
%&\leq&\left(1-\frac{\rho}{\bar{\beta}}\right)\left[\widetilde{\dist}^2(z^{k},\setS_{\mbox{\tiny VIS}})+\beta_2\|u^{k-1}-u^{k}\|^2\right]\\
&\leq&\left(1-\frac{\rho}{\bar{\beta}}\right)^k\left[\widetilde{\dist}^2(z^{1},\setS_{\mbox{\tiny VIS}})+\beta_2\|u^{0}-u^{1}\|^2\right].
\end{eqnarray*}
Then,
\begin{eqnarray*}
\widetilde{\dist}(z^{k+1},\setS_{\mbox{\tiny VIS}})\leq\hat{M}\left(\sqrt{1-\frac{\rho}{\bar{\beta}}}\right)^{k}\quad\mbox{with}\;\hat{M}=\sqrt{\widetilde{\dist}^2(z^{1},\setS_{\mbox{\tiny VIS}})+\beta_2\|u^{0}-u^{1}\|^2}.
\end{eqnarray*}
This shows that the sequence $\{z^k\mid k=1,2,...\}$ converges to the saddle point set $\setS_{\mbox{\tiny VIS}}$ at a linear rate. To be precise,
\[
\lim_{k\rightarrow\infty}\sup\sqrt[k]{\widetilde{\dist}(z^{k+1},\setS_{\mbox{\tiny VIS}})}=\sqrt{1-\frac{\rho}{\bar{\beta}}}<1.
\]
%Then the R-linear of the sequence $\{z^k\}$ can be expressed as in the following corollary.
Formally, we have
\begin{corollary} [R-linear rate of $\{z^k\}$] Suppose the conditions of Theorem~\ref{theo:linear_converg} hold. Then, the sequence $\{z^k\mid k=1,2,...\}$ with $z^k=\left(\begin{array}{c}v^k\\p^{k-1}\end{array}\right)$ converges to the saddle point set $\setS_{\mbox{\tiny VIS}}$ at a linear rate:
\[
\lim_{k\rightarrow\infty}\sup\sqrt[k]{\widetilde{\dist}(z^{k+1},\setS_{\mbox{\tiny VIS}})}=\sqrt{1-\frac{\rho}{\bar{\beta}}}<1.
\]
\end{corollary}

%Next, we give certain instances with the metric subregularity holding.
We shall comment that the metric subregularity condition indeed holds in some commonly encountered environments; cf.~the following proposition.

\begin{proposition} Consider (VIP). Suppose Assumptions~\ref{assumpA} holds. Let $w^*=(u^*, u^*, p^*)$ with $(u^*,p^*)$ be the saddle point of (VIP). Then, the following hold:
\begin{itemize}
\item[{\rm(i)}] If $G(u)$ and $\partial J(u)$ are piecewise linear, $\mathbf{U}$ is polyhedral, $\Theta(u)=Au-b$, and $\mathbf{C}=\{0\}$, then $H(w)$ is metric subregular around $(w^*,0)$.
\item[{\rm(ii)}] If $G(u)=Qu$, $J(u)=\langle c, u\rangle ,\; Q\in \RR^{n\times n}$ is symmetric p.s.d.\ matrix, $c\in \RR^n$, $\mathbf{U}$ is polyhedral, $\Theta(u)=Au-b$, and $\mathbf{C}$ is polyhedral convex cone in $\RR^m$, then $H(w)$ is metric subregular around $(w^*,0)$.
\end{itemize}
\end{proposition}

\begin{proof} {\ }
\begin{itemize}
\item[(i)] Cf.~Theorem 3.3 of~\cite{ZN14}.
\item[(ii)] Cf.~Proposition 1 and Corollary of~\cite{R81}.
\end{itemize}
\end{proof}
%\backmatter
%
%
%%
%

%\section{Experimental Results}
\section{Numerical Experiments} 
\label{exp}

In this section, we shall test the performance of ALAVI on %implementation for solving 
two sets of randomly generated {\it highly}\/ non-monotone and constrained Variational Inequality (N-CVI) problem instances specified in the following two subsections.

\subsection{Non-monotone and constrained variational inequality problems experiment results I (N-CVI-1)}

By constructing the mapping $G(u):=M(u)(u-\mathbf{1}_n/4)$, where $M:\RR^n\rightarrow\RR^{n\times n}$ is a matrix-valued function and $M(u)$ is a symmetric positive semidefinite matrix for any $u$, our objective is to find a solution to the following highly non-monotone constrained variational inequality: 
\begin{equation}\label{eq:N_CVI}
\mbox{\rm(N-CVI-1)}\quad\langle G(u^*),u-u^*\rangle\geq0,\quad\forall u\in\Gamma,
\end{equation}
where $\Gamma=\{u\in[0,1]^n \mid \mathbf{1}_n^{\top}u-\frac{n}{2}\leq0\}$. Moreover, introduce $M$ as
\[
M(u):= t_1t_1^{\top}+t_2t_2^{\top},\quad\mbox{with}\quad t_1=A\cos u,\quad t_2 = B\cdot\frac{\mathbf{1}_n}{\mathbf{1}_n+\exp(-u)},
\]
where $u\in\RR^n$, $A, B\in\RR^{n\times n}$ and the operations `$\cos$' and `$\exp$' 
are taken componentwise. 
%should be understood to apply entrywise.

%We note that if 
To see that $G$ is non-monotone, consider e.g.\ $n=2$, $A=\left(\begin{array}{cc}-0.9&-0.8\\0.3&1.2\end{array}\right)$, $B=\left(\begin{array}{cc}0.9&0.7\\-0.3&-0.3\end{array}\right)$, $u_1=\left(\begin{array}{c}0\\0.7\end{array}\right)$, and $u_2=\left(\begin{array}{c}0.1\\0.9\end{array}\right)$, then $\langle G(u_1)-G(u_2),u_1-u_2\rangle=-0.0245<0$. %which shows that $G$ is non-monotone. 
To see that Assumption~\ref{AssumptionB} does hold, we let $u^{\sharp}=\mathbf{1}_n/4$ and $p^{\sharp}=0$. We have that $\mathbf{1}_n^{\top}u^{\sharp}-n/2=-n/4\leq0$, leading to $\langle G(u),u-u^{\sharp}\rangle+\langle p^{\sharp}, \mathbf{1}_n^{\top}(u-u^{\sharp})\rangle=\langle G(u),u-u^{\sharp}\rangle=M(u)(u-\mathbf{1}_n/4)^2\geq0$, satisfying Assumption~\ref{AssumptionB} and so ALAVI is guaranteed to converge. 
%. Then the primal-dual variational coherence condition holds, and we can use ALAVI to solve~\eqref{eq:N_CVI}.

The ALAVI scheme for (N-CVI-1) can be explicitly expressed as follows:
\[
\left\{\begin{array}{l}
v^k:=(1-\eta)u^k+\eta v^{k-1}; \\ q^k:=\max\left(0,p^k+\gamma(\mathbf{1}_n^{\top}u^k-n/2)\right); \\ 
u^{k+1}:={\rm proj}_{[0,1]^n}\left(0,v^k-\alpha(G(u^k)+\mathbf{1}_nq^k)\right); \\ p^{k+1}:=\max\left(0,p^k+\gamma(\mathbf{1}_n^{\top}u^{k+1}-n/2)\right).\end{array}\right.
\]
We consider the following settings: $n\in\{100,500,1000,2000\}$, and the entries of matrices $A$ and $B$ are independently drawn from the normal distribution $\mathbb{N}(0, 1)$. The initial point for ALAVI is always set as $u^0 = \mathbf{1}_n$. As the function $G$ is non-monotone, there may exist multiple solutions to (N-CVI-1). Therefore, we utilize the KKT error, defined as ${\rm dist}\left(0,G(u^k)+\mathbf{1}_np^k+\mathcal{N}_{[0,1]^n}(u^k)\right)+\|\max\{0,\mathbf{1}_n^{\top}u^{k}-n/2\}\|$, to evaluate the performance of ALAVI. Figure~\ref{fig:num1} illustrates the KKT errors plotted against the iteration counts. The results demonstrate that ALAVI effectively and efficiently solves the (N-CVI-1) instances. 
\begin{figure}[ht]
\begin{center}
{\includegraphics[width=0.24\textwidth]{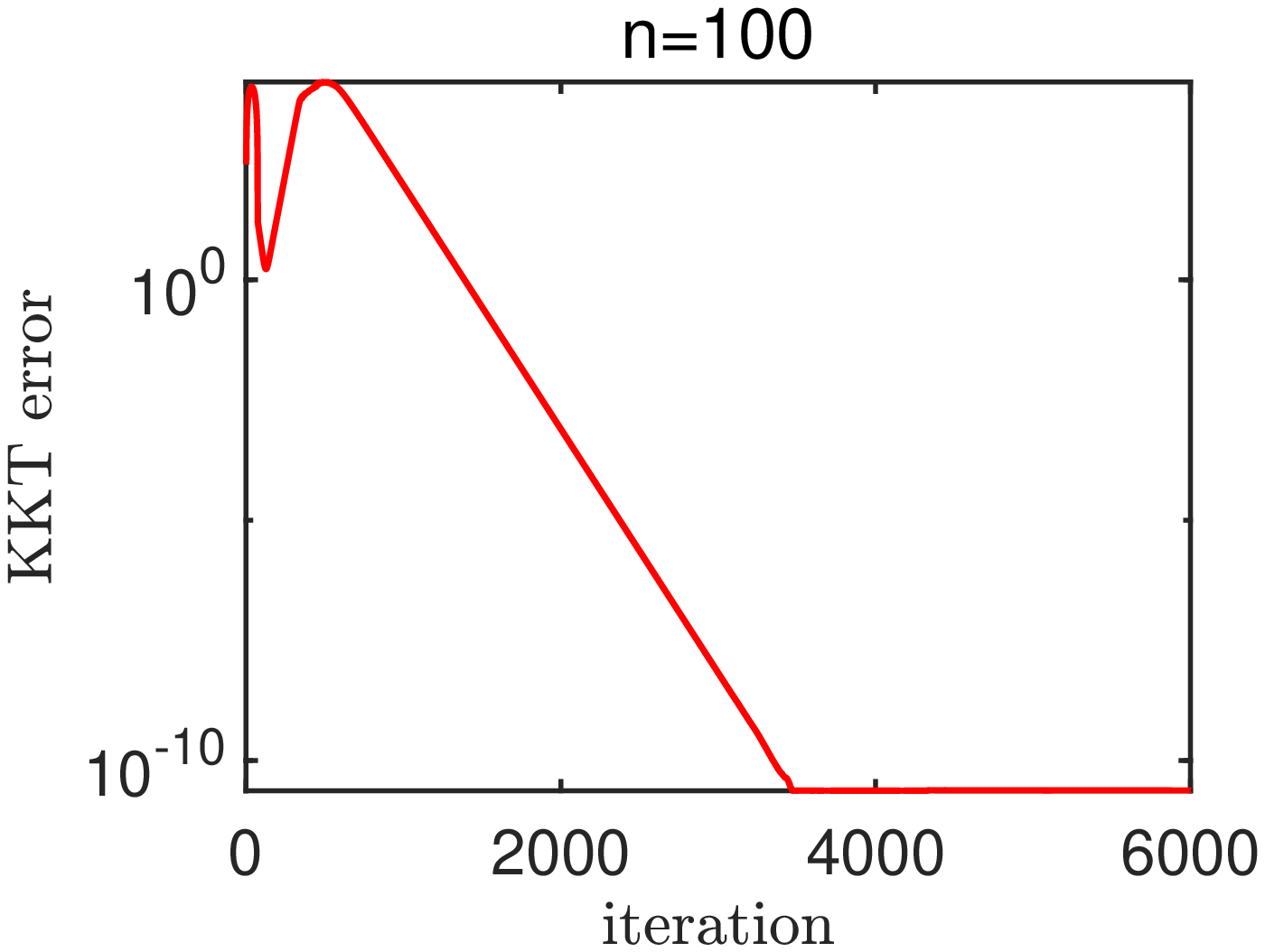}}
{\includegraphics[width=0.24\textwidth]{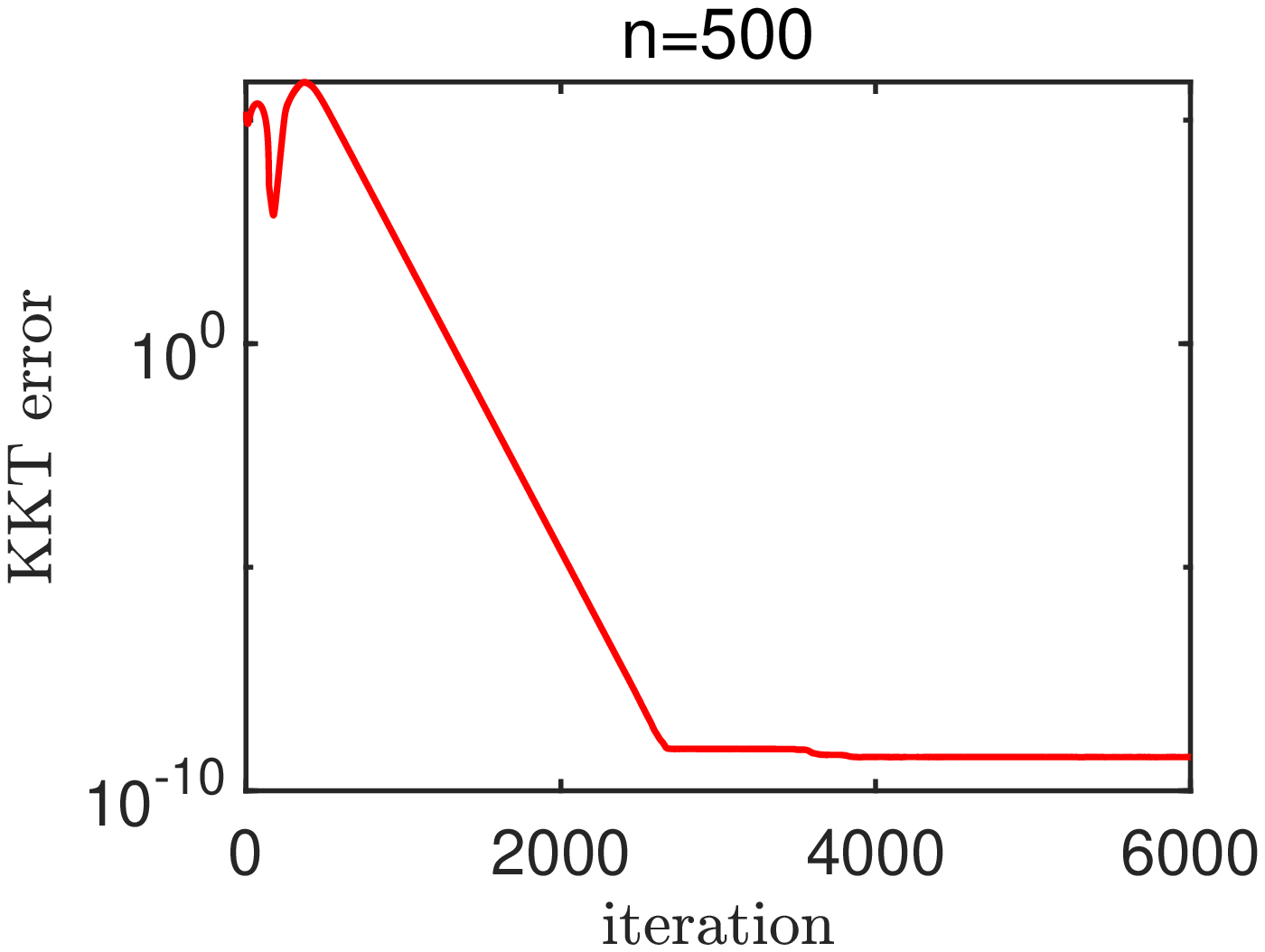}}
{\includegraphics[width=0.24\textwidth]{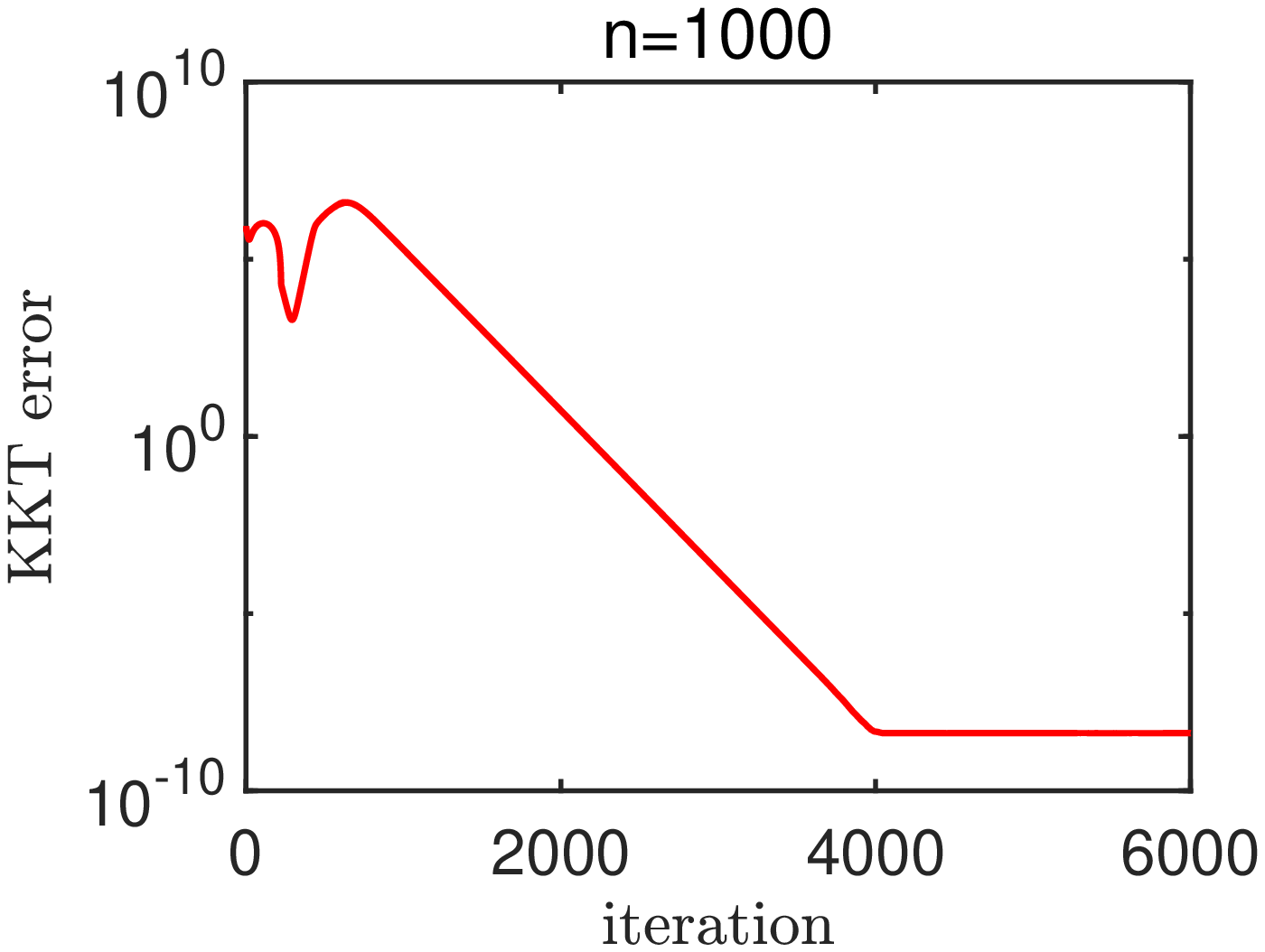}}
{\includegraphics[width=0.24\textwidth]{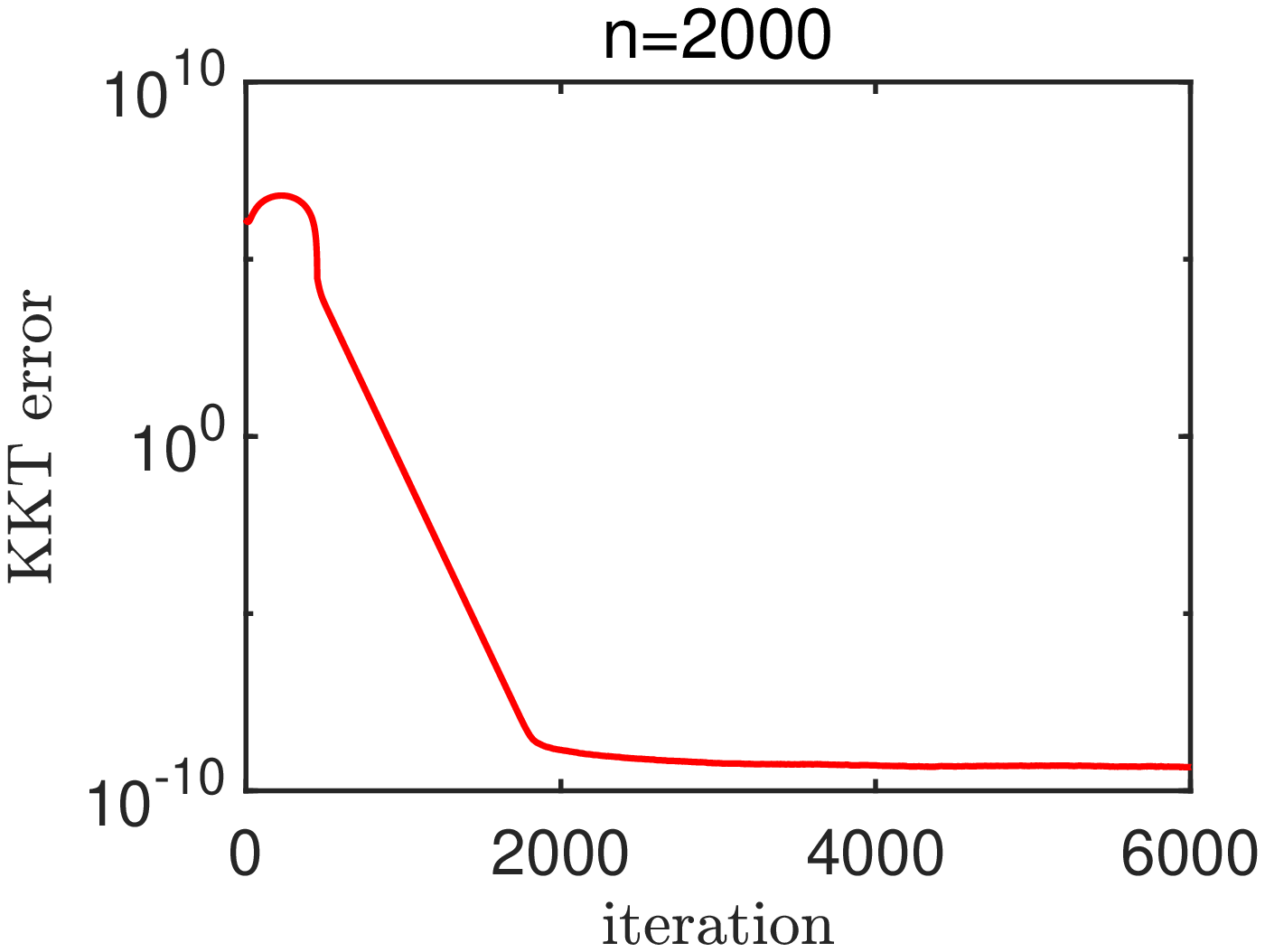}}
\vskip -0.1in
\caption{Convergence of ALAVI on (N-CVI-1)}\label{fig:num1}
\end{center}
\vspace{-1cm}
\end{figure}

\subsection{Non-monotone and constrained variational inequality problems experiment results II (N-CVI-2)}

%{Non-monotone constrained variational inequality-2 (N-CVI-2)}
In this subsection, we shall experiment with another type of general highly nonlinear and non-monotone VI problems where the constraint set is a general polyhedron. Specificcally, we define the mapping $G(u)=D(u)^{\circ2}\circ(u-u^{\sharp})$ where $D:\RR^n\rightarrow\RR^{n}$, $u=(u_1,...,u_n)^{\top}\in\RR^n$, $u^{\sharp}$ is a given point, $\circ2$ represents the Hadamard power and $\circ$ denotes the Hadamard product. In this construction, $G$ is highly nonlinear and non-monotone. Specifically, in this subsection, we set $D(u)=(u_n,...,u_1)^{\top}$ and $J(u)=\|u-\mathbf{1}_n\|_1$. Our objective is to find a solution to the following non-monotone constrained variational inequality:
\begin{equation}\label{eq:N_CVI_2}
\mbox{\rm(N-CVI-2)}\quad\langle G(u^*),u-u^*\rangle+J(u)-J(u^*)\geq0,\quad\forall u\in\Gamma,
\end{equation}
where $\Gamma=\{u\in[-10,10]^n \mid Au-b\leq0\}$ with $A\in\RR^{m\times n}$ and $b\in\RR^m$.

First, we note $G$ is non-monotone: for $n=2$ consider $u=\left(\begin{array}{c}1\\ \epsilon\end{array}\right)$, and $u'=\left(\begin{array}{c}\epsilon\\1\end{array}\right)$ with $\epsilon>1$. Then,
\begin{eqnarray*}
\langle G(u)-G(u'),u-u'\rangle&=&\left\langle\left(\begin{array}{c}\frac{1}{2}\epsilon^2-\epsilon+\frac{1}{2}\\ \epsilon-\frac{1}{2}-\frac{1}{2}\epsilon^2\end{array}\right),\left(\begin{array}{c}1-\epsilon\\ \epsilon-1\end{array}\right)\right\rangle\\
&=&\frac{1}{2}(1-\epsilon)^3-\frac{1}{2}(\epsilon-1)^3
%\\
= - (\epsilon-1)^2 \overset{\epsilon>1}{<} 0,
\end{eqnarray*}
indicating that $G$ is non-monotone in general. Second, we choose  $u^{\sharp}$ as used in $G$, and the vector $p^{\sharp}\in\RR^m$, to be the optimal primal and dual (the associated Lagrangian multiplier) solutions of the convex optimization problem: 
$\min_{u\in\Gamma}J(u)$,
which ensures that %Then we can easily obtain that
\[
J(u)-J(u^{\sharp})+\langle p^{\sharp},A(u-u^{\sharp})\rangle\geq0,
\]
and
\[
\langle p-p^{\sharp},Au^{\sharp}-b\rangle\leq0.
\]
It follows that
\[
\langle G(u),u-u^{\sharp}\rangle+J(u)-J(u^{\sharp})+\langle p^{\sharp}, A(u-u^{\sharp})\rangle\geq\langle D(u)(u-u^{\sharp}),u-u^{\sharp}\rangle\geq0.  
\]
The above ensures that the primal-dual variational coherence condition holds, and ALAVI is guaranteed to converge for solving (N-CVI-2).
In fact, the ALAVI scheme for (N-CVI-2) can be explicitly written as
\[
\left\{\begin{array}{l}
v^k:=(1-\eta)u^k+\eta v^{k-1}; \\ q^k:=\max\left(0,p^k+\gamma(Au^k-b)\right); \\ u^{k+1}:=\arg\min_{u\in[-10,10]^n}\langle G(u^k)+A^{\top}q^k,u\rangle +J(u)+\frac{1}{2\alpha}\|u-v^k\|^2;\\ 
p^{k+1}:=\max\left(0,p^k+\gamma(Au^{k+1}-b)\right).\end{array}\right.
\]
We consider the following settings: $n\in\{100,500,1000,2000\}$, $m=n/50$, the entries of matrix $A$ are independently drawn from the normal distribution $\mathbb{N}(0, 1)$, and $b=A\cdot(\mathbf{1}_n/2)$. The starting point for ALAVI is always $u^0 = \mathbf{1}_n$. As the mapping $G$ is non-monotone, there may exist multiple solutions to (N-CVI-2). Therefore, we measure the performance of ALAVI by its the error with regard to the KKT system, defined as ${\rm dist}\left(0,G(u^k)+\partial J(u^k)+A^{\top}p^k+\mathcal{N}_{[-10,10]^n}(u^k)\right)+\|\max\{0,Au^{k}-b\}\|$. Figure~\ref{fig:num2} displays the KKT error plotted against the iteration counts, which shows  
%The results depicted in Figure~\ref{fig:num2} demonstrate 
that ALAVI effectively and easily solves (N-CVI-2).
\begin{figure}[ht]
\begin{center}
{\includegraphics[width=0.24\textwidth]{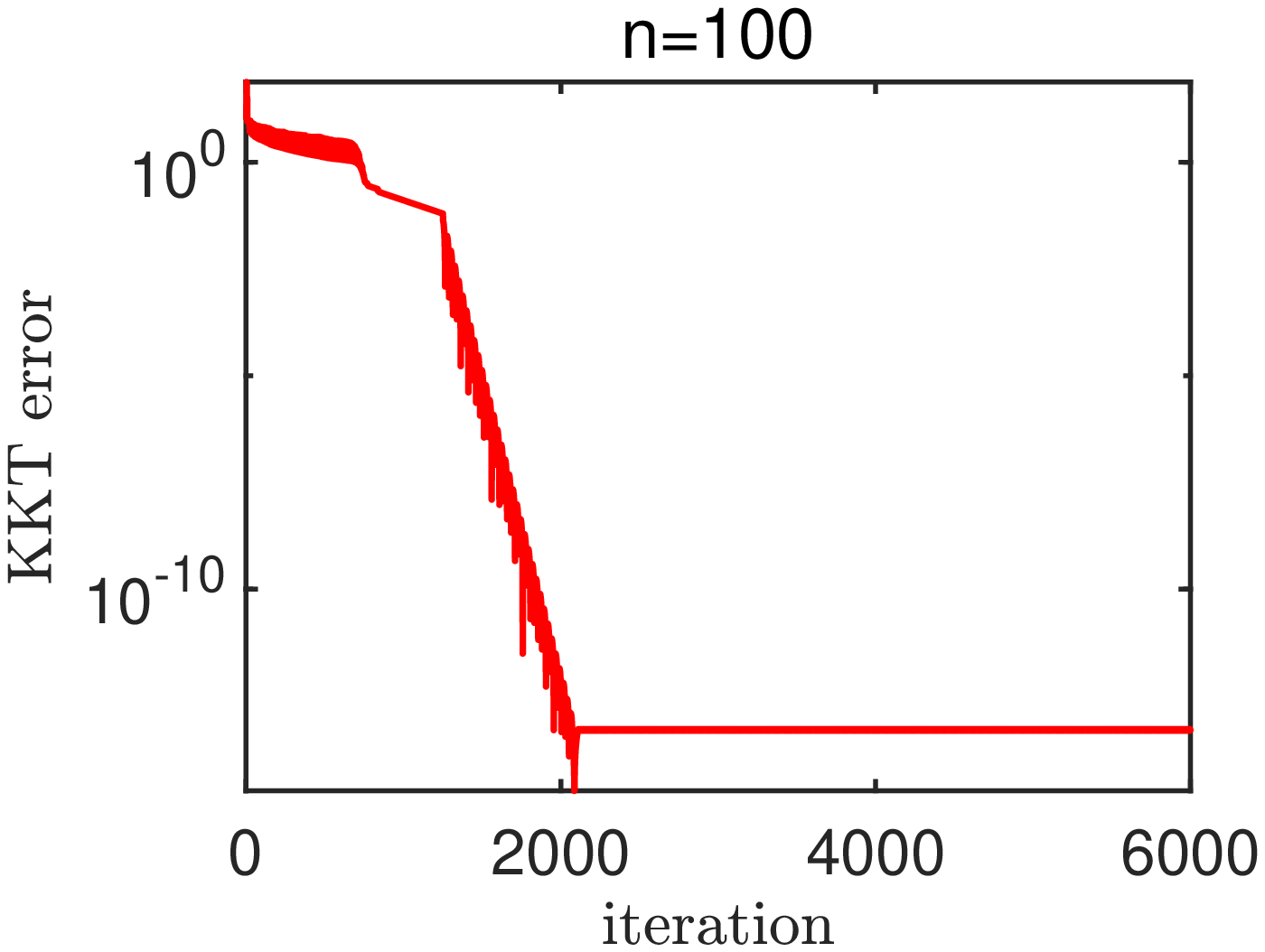}}
{\includegraphics[width=0.24\textwidth]{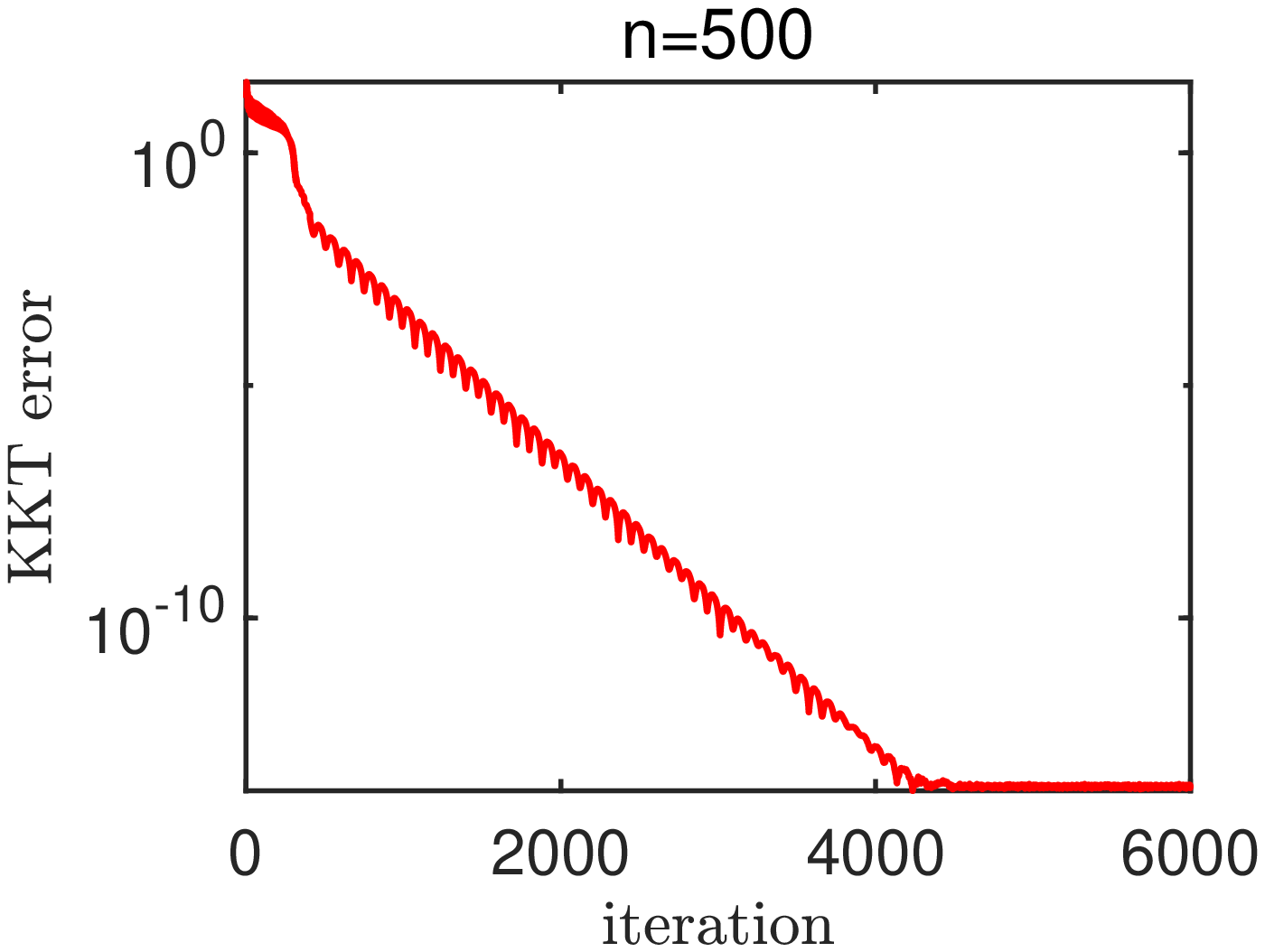}}
{\includegraphics[width=0.24\textwidth]{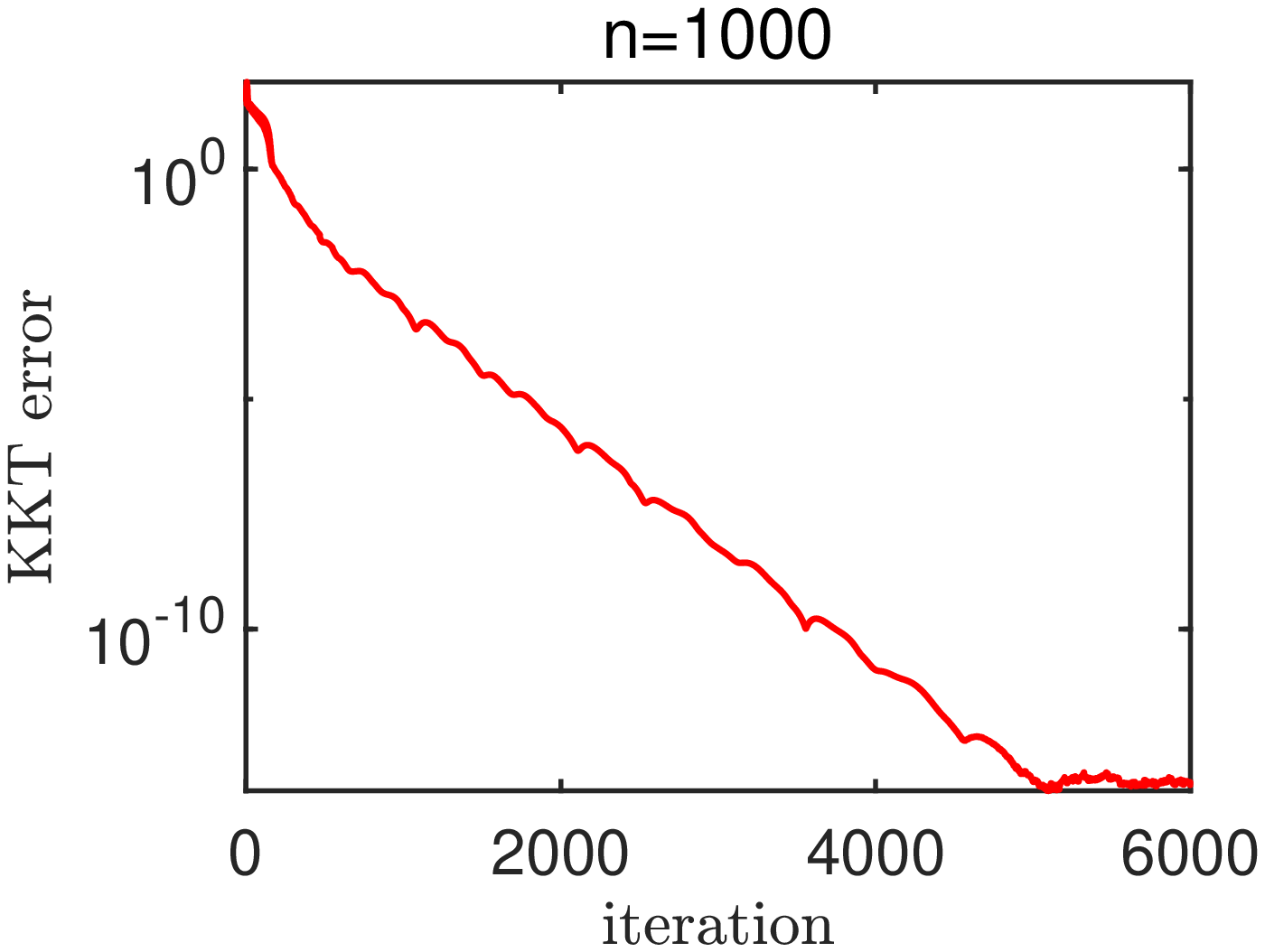}}
{\includegraphics[width=0.24\textwidth]{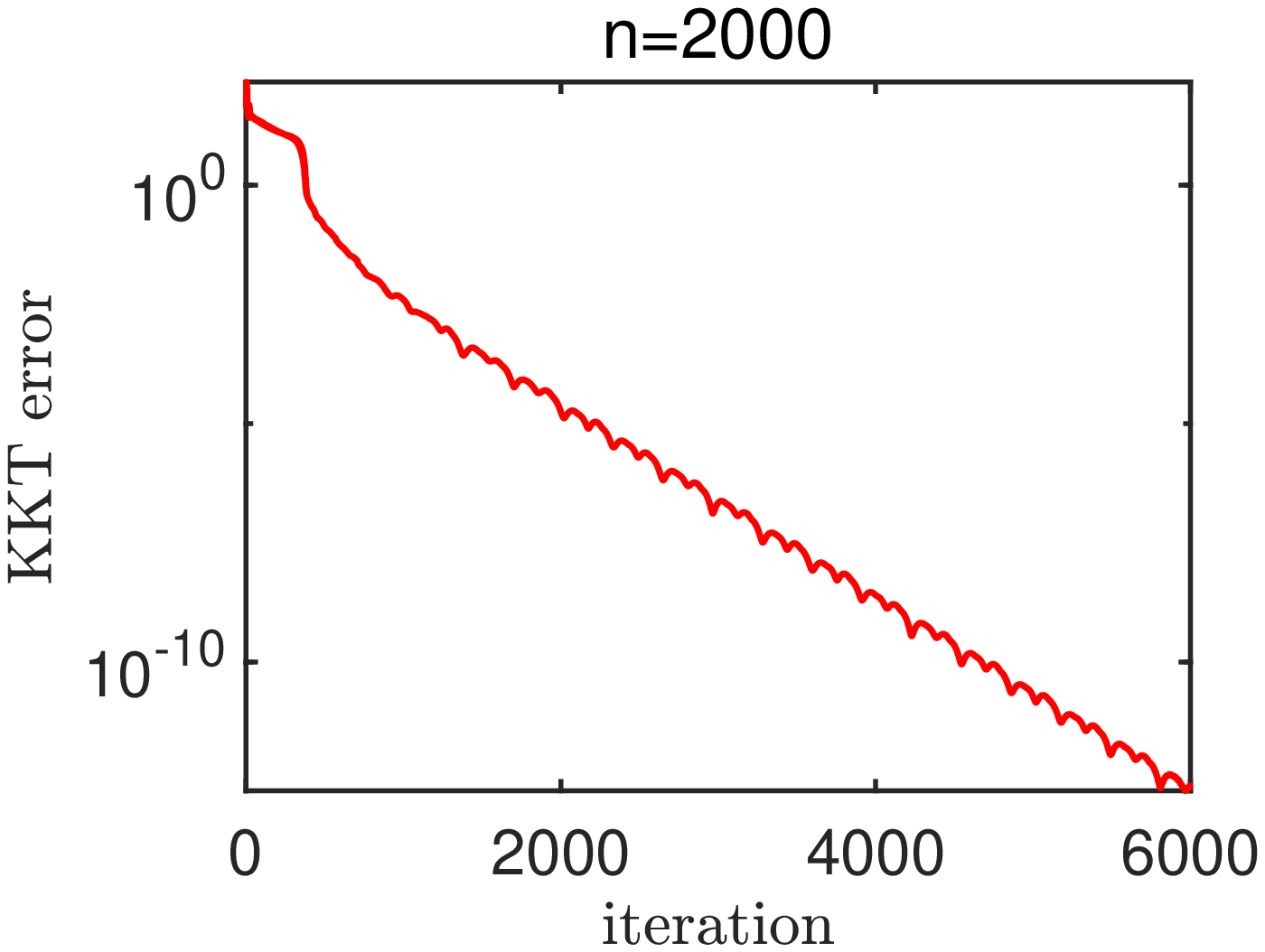}}
\vskip -0.1in
\caption{Convergence of ALAVI on (N-CVI-2)}\label{fig:num2}
\end{center}
\vspace{-1cm}
\end{figure}

{\bf Acknowledgments.}
The research of the first two co-authors was supported by NSFC grant 71871140 and 72293582.

\appendix
\section{Examples to Illustrate the Subclasses Discussed in Section~\ref{sec:monotone}} \label{appendix}

\begin{example}[($J+\langle p^*, \Theta \rangle )$-pseudo monotone]\label{exap1}
Consider (VIP) with $G(u)=\frac{1}{1+u},\;J(u)=u$, $\UU =\RR^+,$ and constraint $\Theta (u)=u-1\leq 0$.
\begin{itemize}
\item [{\rm(i)}] {\bf Saddle point.} Consider the following two variational inequality systems
\begin{eqnarray}
\mbox{\rm (VIS)} \qquad &&\left\langle \frac{1}{1+u^*}, u-u^*\right\rangle +(u-u^*)+\langle p^*,
u-u^* \rangle \geq 0,  \quad \forall u\in\UU \nonumber \\
&&\langle  u^*-1, p-p^*\rangle \leq 0, \quad \forall p \geq 0 \nonumber
\end{eqnarray}
and
\begin{eqnarray}
\mbox{\rm (MVIS)} \qquad &&\left\langle \frac{1}{1+u}, u-u^\natural \right\rangle +(u-u^\natural )+\langle p^\natural ,
u-u^\natural \rangle \geq 0,  \quad \forall u\in \UU  \nonumber \\
&&\langle  u^\natural -1, p-p^\natural \rangle \leq 0, \quad \forall p \geq 0 . \nonumber
\end{eqnarray}
Obviously, $(0,0)$ is a saddle point of (VIS) and (MVIS).
\item [{\rm(ii)}] {\bf Star monotone and monotone.} Since $\langle G(u)-G(u^*), u-u^*\rangle =\langle \frac{1}{1+u}-1, u-0\rangle \leq 0, \;\forall u\geq 0$, we know that
$G(u)$ is not star monotone at $u^*$ on $\UU$ and is non-monotone on $\UU$.
\item [{\rm(iii)}] {\bf $(J+\langle p^*, \Theta \rangle)$-pseudo monotone and $J$-pseudo monotone.} For $u,v\in \UU$ if $\langle G(v), u-v\rangle +J(u)-J(v)+\langle p^*, u-v\rangle\geq 0$, i.e.,
\begin{eqnarray}
\qquad &&\left\langle \frac{1}{1+v}, u-v\right\rangle +(u-v)=(u-v)\left[\frac{1}{1+v}+1 \right]\geq 0 \implies u\geq v . \nonumber
\end{eqnarray}
Thus,
\begin{eqnarray}
\qquad &&\left\langle \frac{1}{1+u}, u-v\right\rangle +(u-v)=(u-v) \left[\frac{1}{1+u}+1 \right]\geq 0 , \mbox{i.e., } \nonumber\\
       &&\langle G(u),u-v\rangle +J(u)-J(v)+\langle p^*, u-v\rangle \geq 0,  \mbox{with } p^*=0 . \nonumber
\end{eqnarray}
Hence, $G$ is $ (J+\langle p^*, \Theta \rangle )$-pseudo monotone on $\UU$. Obviously, $G$ is also $J$-pseudo monotone.
\item [{\rm(iv)}] {\bf Pseudo monotone.} For $u,v\in \UU$,  $\langle G(v), u-v\rangle =\langle \frac{1}{1+v}, u-v\rangle \geq 0 \implies u\ge v$. \\
Hence,
$\langle G(u), u-v\rangle =\langle \frac{1}{1+u}, u-v\rangle \geq 0 $. Therefore, $G$ is pseudo monotone.
\end{itemize}
\end{example}
\begin{example}[Both star monotone and $(J+\langle p^*, \Theta \rangle )$-pseudo monotone]\label{exap2}
Consider (VIP) with $G(u)=\sin u -1,\;J=u$, $\UU=[0, \pi ]£¬$ and constraint $\Theta (u)=u-\frac{3 \pi}{4} \leq 0$
\begin{itemize}
\item [{\rm(i)}] {\bf Saddle point.} Consider the following two variational inequality systems
\begin{eqnarray}
\mbox{\rm (VIS)} \qquad &&\langle (\sin u^*-1 ), u-u^*\rangle +(u-u^*)+\langle p^*,
u-u^*\rangle \geq 0,  \quad \forall u\in\UU  \nonumber\\
&& \left\langle  u^*-\frac{3\pi }{4}, p-p^* \right\rangle \leq 0, \quad \forall p \geq 0 \nonumber
\end{eqnarray}
and
\begin{eqnarray}
\mbox{\rm (MVIS)} \qquad &&\langle (\sin u-1 ), u-u^\natural \rangle +(u-u^\natural )+\langle p^\natural ,
u-u^\natural \rangle \geq 0,  \quad \forall u\in\UU \nonumber \\
&& \left\langle  u^\natural -\frac{3\pi }{4}, p-p^\natural \right\rangle \leq 0, \quad \forall p \geq 0. \nonumber
\end{eqnarray}
Obviously, $(0,0)$ is a saddle point of (VIS) and (MVIS).
\item [{\rm(ii)}] {\bf Star monotone and monotone.} Since $\langle G(u)-G(u^*), u-u^*\rangle =\langle \sin u, u\rangle \geq 0, \;\forall u\in [0,\pi ]$, then
$G(u)$ is star monotone at $u^*$ on $\UU$. Taking $u=\frac{\pi }{2},\;v=\frac{3\pi}{4}$, we have that $\langle G(u)-G(v), u-v\rangle <0$. Therefore, $G$  is non-monotone on $\UU$.
\item [{\rm(iii)}] {\bf $(J+\langle p^*, \Theta \rangle )$-pseudo monotone and $J$-pseudo monotone.} For $u,v\in \UU$ if $\langle G(v), u-v\rangle +J(u)-J(v)\geq 0$, i.e., $\sin v \,(u-v)\geq 0$, then $u\geq v$.

Therefore, $\langle G(u), u-v\rangle +J(u)-J(v)=\sin u \,(u-v)\geq 0$ and $G$ is
$ (J+\langle p^*, \Theta \rangle )$-pseudo monotone on $\UU$ with $J(u)=u,\;p^*=0$. Obviously $G$ is also $J$-pseudo monotone.
\item [{\rm(iv)}] {\bf Pseudo monotone.} Taking $u=\frac{\pi}{2},\; v\in (\frac{\pi}{2}, \frac{3\pi}{4})$, we have that
$\langle G(u), v-u\rangle =0 $ and $\langle G(v), v-u\rangle < 0 $. Therefore, $G$ is not pseudo monotone on $\UU$.
\end{itemize}
\end{example}
\begin{example}[Both quasi monotone and $(J+\langle p^*, \Theta \rangle )$-quasi monotone]\label{exap3}
Consider (VIP) with $G(u)=u^2$, $J(u)=0$, $\UU=[-1,1]$ and constraint $\Theta (u)=u\leq 0$.
\begin{itemize}
\item [{\rm(i)}] {\bf Saddle point.} Consider the following two variational inequality systems
\begin{eqnarray}
\mbox{\rm (VIS)} \qquad &&\langle (u^*)^2, u-u^*\rangle +\langle p^*,
u-u^*\rangle \geq 0,  \quad \forall u\in\UU \nonumber \\
&&\langle u^*, p-p^*\rangle \leq 0, \quad \forall p \geq 0 \nonumber
\end{eqnarray}
and
\begin{eqnarray}
\mbox{\rm (MVIS)} \qquad &&\langle u^2, u-u^\natural \rangle +\langle p^\natural ,
u-u^\natural \rangle \geq 0,  \quad \forall u\in\UU  \nonumber \\
&&\langle  u^\natural , p-p^\natural \rangle \leq 0, \quad \forall p \geq 0 \nonumber
\end{eqnarray}
and $(-1, 0)$ is one saddle point of (VIS) and (MVIS).
\item [{\rm(ii)}] {\bf Star monotone and monotone.} For $u^*=-1$, since $\langle G(u)-G(u^*), u-u^*\rangle =-1 < 0, $
 whenever $u=0$. Thus, $G$ is not star monotone at $u^*$ on $\UU$.
 \item [{\rm(iii)}] {\bf $(J+\langle p^*, \Theta \rangle)$-pseudo monotone and $J$-pseudo-monotone.} For $p^*=0$, $u=1$ and $v=0$, we have that $\langle G(v), v-u\rangle =v^2(v-u)=0$ and $\langle G(u), v-u\rangle =u^2(v-u)<0$. Therefore, $G$ is not $J$-pseudo monotone and isn't $ (J+\langle p^*, \Theta \rangle )$-pseudo monotone on $\UU$ with $J(u)=0$ and $p^*=0$.
\item [{\rm(iv)}] {\bf Pseudo monotone.} By (iii) of this example, $G$ is not pseudo monotone on $\UU$ either.
\item [{\rm(v)}] {\bf Quasi monotone and $(J+\langle p^*, \Theta \rangle )$-quasi monotone.} For $u,\;v\in \UU$, if $\langle G(v),v-u\rangle >0$, then $v\neq 0$ and $v>u$. Therefore, we have $\langle G(u),v-u\rangle \geq 0$, which shows $G$ is quasi monotone on $\UU$. In this case, $J=0$ and $p^*=0$. Thus, $G$ is also $(J+\langle p^*, \Theta \rangle )$-quasi monotone.
\end{itemize}
\end{example}
\begin{example}[star monotonicity]\label{exap4}
Consider (VIP) with $u=\left(\begin{array}{c}x\\y\end{array}\right)$, $G(u)=G(x,y)=\left(\begin{array}{c}2x\left( y^2+1\right)\\2y\left(x^2+1\right)\end{array}\right)$, $J(u)=J(x,y)=x+y$, $\Theta (u)=x-y=0$ and $\UU =\RR^2_+$.
\begin{itemize}
\item [{\rm(i)}] {\bf Saddle point.} Consider the following two variational inequality systems
\begin{eqnarray*}
\mbox{\rm (VIS)} && \left\langle \left(\begin{array}{c}2x^*\left({y^*}^2+1\right)\\2y^*\left({x^*}^2+1\right)\end{array}\right), \left(\begin{array}{c}x-x^*\\y-y^*\end{array}\right)\right\rangle +(x+y)-(x^*+y^*)\\
&&\qquad\qquad\qquad\qquad\qquad\quad+\langle p^*,
(x-y)-(x^*-y^*)\rangle \geq 0,  \; \forall u\in\UU\\
&&\langle  x^*-y^*, p-p^*\rangle \leq 0, \quad \forall p \in\RR
\end{eqnarray*}
and
\begin{eqnarray*}
\mbox{\rm (MVIS)} \qquad && \left\langle \left(\begin{array}{c}2x\left({y}^2+1\right)\\2y\left({x}^2+1\right)\end{array}\right), \left(\begin{array}{c}x-x^\natural \\y-y^\natural \end{array}\right) \right\rangle +(x+y)-(x^\natural +y^\natural )\\
&&\qquad\qquad\qquad\quad+\langle p^\natural ,
(x-y)-(x^{\natural}-y^{\natural})\rangle \geq 0,  \quad \forall u\in\UU \\
&&\langle  x^\natural -y^\natural , p-p^\natural \rangle \leq 0, \quad \forall p \in\RR
\end{eqnarray*}
where $u^*=\left(\begin{array}{c}0\\0\end{array}\right)$, $p^*=1$ is a saddle point of (VIS) and (MVIS).
\item [{\rm(ii)}] {\bf Star monotone and monotone.} Let $u^*=\left(\begin{array}{c}0\\0\end{array}\right)$. For any $u\in\RR^2_+$, we have
 $\langle G(u)-G(u^*), u-u^*\rangle =\left\langle \left(\begin{array}{c}2x\left({y}^2+1\right)\\2y\left({x}^2+1\right)\end{array}\right),
 \left(\begin{array}{c}x\\y\end{array}\right)\right\rangle =2x^2(y^2+1)+2y^2(x^2+1)\geq 0$.
  Thus, $G$ is star monotone at $u^*$ on $\UU$.\\
  %For monotonicity of $G$, let 
  Consider $u=\left(\begin{array}{c}x\\y\end{array}\right)=\left(\begin{array}{c}1\\6\end{array}\right)$ and $u'=\left(\begin{array}{c}x'\\y'\end{array}\right)=\left(\begin{array}{c}3\\3\end{array}\right)$. We have $G(u)=G(x,y)=\left(\begin{array}{c}74\\24\end{array}\right)$, $G(u')=G(x',y')=\left(\begin{array}{c}60\\60\end{array}\right)$,
  $u-u'=\left(\begin{array}{c}-2\\3\end{array}\right)$ and $\langle G(u)-G(u'), u-u'\rangle =\left\langle \left(\begin{array}{c}14\\-36\end{array}\right), \left(\begin{array}{c}-2\\3\end{array} \right )\right\rangle <0$, which shows that $G$ is non-monotone.
\item [{\rm(iii)}] {\bf $(J+\langle p^*, \Theta \rangle )$-pseudo monotone and $J$-pseudo monotone.} Take the same two points $u$ and $u'$ as in (ii). We have $J(u)-J(u')=1$ and $\Theta (u)-\Theta (u')=-5$. For $p^*=1$, we have $$\langle G(u'), u-u'\rangle +J(u)-J(u')+\langle p^*, \Theta (u)-\Theta (u')\rangle =56\geq 0.$$
  However,
  \[
  \langle G(u), u-u'\rangle +J(u)-J(u')+\langle p^*, \Theta (u)-\Theta (u')\rangle =-80 <0.
  \]
  Therefore, $G$ is not $ (J+\langle p^*, \Theta \rangle )$-pseudo monotone on $\UU$. We also find that $G$ is not $J$-pseudo monotone on $\UU$.
\item [{\rm(iv)}] {\bf Pseudo monotone.}  Again, take the same two points as in (ii): $u=\left(\begin{array}{c}x\\y\end{array}\right)=\left(\begin{array}{c}1\\6\end{array}\right)$ and $u'=\left(\begin{array}{c}x'\\y'\end{array}\right)=\left(\begin{array}{c}3\\3\end{array}\right)$. We have
  $\langle G(u'), u-u'\rangle =60\geq 0$. However, $\langle G(u), u-u'\rangle =-76 <0$. Therefore, $G$ is not peudo monotone on $\UU =\RR_+^2$.
\end{itemize}
\end{example}

\begin{remark}{\ }

\begin{itemize}
\item [{\rm(i)}] From Examples~\ref{exap1} and~\ref{exap2}, we see that a star monotone mapping is not necessarily pseudo monotone. Conversely, a pseudo monotone mapping is not necessarily star monotone. On the other hand, a $(J+\langle p^*, \Theta \rangle )$-pseudo monotone mapping is not necessarily pseudo monotone.

\item [{\rm(ii)}] For the case $p^*=0$, the $(J+\langle p^*, \Theta \rangle )$-pseudo monotonicity coincides with the $J$-pseudo monotonicity, and
the $(J+\langle p^*, \Theta \rangle )$-quasi monotonicity coincides with the $J$-quasi monotonicity.

\item [{\rm(iii)}] When $J=0$, the $J$-pseudo monotonicity of $G$ coincides with the pseudo monotonicity of $G$, and
the $J$-quasi monotonicity of $G$ coincides with the quasi monotonicity of $G$.
 \end{itemize}
\end{remark}

\end{document}